\documentclass[oneside, book,11 pt]{amsart}
\usepackage{color}
\usepackage{cancel}
\usepackage{amsmath}
\usepackage{amssymb}
\usepackage[normalem]{ulem}
\usepackage{amsfonts}
\usepackage{bbm}
\usepackage{amsthm}
\usepackage{enumerate}
\usepackage[mathscr]{eucal}
\usepackage{graphicx}
\usepackage{pstricks}
\usepackage[active]{srcltx}
\usepackage{fancyhdr}
\usepackage{verbatim}
\usepackage{fullpage}
\usepackage{hyperref}
\hypersetup{
colorlinks = true, 
urlcolor = blue, 
linkcolor = red, 
citecolor = blue 
}
\usepackage{caption}
\usepackage{float}

\newtheorem{thm}{Theorem}[section]
\newtheorem*{thm*}{Theorem}
\newtheorem{lemma}[thm]{Lemma}
\newtheorem{corollary}[thm]{Corollary}
\newtheorem{prop}[thm]{Proposition}
\newtheorem*{prop*}{Proposition}

\numberwithin{equation}{section}

\title{Four-Cycle Free Graphs, Height Functions, the Pivot Property and Entropy Minimality}
\author{
Nishant Chandgotia
}\address{
Department of Mathematics\\ University of British Columbia,Canada}
\email {nishant@math.ubc.ca}
\def\F{{\mathcal F}}
\def\A{{\mathcal A}}
\def\N{\mathbb N}

\newcommand{\Z}{\mathbb{Z}}
\def \B{\mathcal B}

\def \E{\mathbb E}
\def \R{\mathbb R}

\def \V{\mathcal V}
\def \E1{\mathcal E}

\def \L{\mathcal L}
\def \T{\mathcal T}
\def\p{\prime}
\def\t{\tilde}
\def \m{\vec}
\def \H{\mathcal H}
\def\Q{\mathbb Q}
\def \C{\mathcal C}
\def\mi{{\vec{i}}}
\def\mj{{\vec{j}}}
\begin{document}
\maketitle

\begin{abstract}
Fix $d\geq 2$. Given a finite undirected graph $\H$ without self-loops and multiple edges, consider the corresponding `vertex' shift, $Hom(\Z^d, \H)$ denoted by $X_\H$. In this paper we focus on $\H$ which is `four-cycle free'. The two main results of this paper are: $X_\H$ has the pivot property, meaning that for all distinct configurations $x,y\in X_\H$ which differ only at finitely many sites there is a sequence of configurations $x=x^1, x^2, \ldots, x^n=y\in X_{\H}$ for which the successive configurations $x^i, x^{i+1}$ differ exactly at a single site. Further if $\H$ is connected then $X_\H$ is entropy minimal, meaning that every shift space strictly contained in $X_\H$ has strictly smaller entropy. The proofs of these seemingly disparate statements are related by the use of the `lifts' of the configurations in $X_\H$ to the universal cover of $\H$ and the introduction of `height functions' in this context.
\end{abstract}

\section{Introduction}\label{section: Introduction}

By $\H$ we will always denote an undirected graph without multiple edges (and by abuse of notation also denote its set of vertices). In this paper we focus on $\H$ which is four-cycle free, that is, it is finite, it has no self-loops and the four-cycle, denoted by $C_4$ is not a subgraph of $\H$. Fix $d\geq 2$. The basic object of study is $X_\H$, the space of graph homomorphisms from $\Z^d$ to $\H$. Here by $\Z^d$ we will mean both the group and its standard Cayley graph.

Such a space of configurations $X_\H$, referred to as a hom-shift, can be obtained by forbidding certain patterns on edges of $\H^{\Z^d}$. If $\H$ is a finite graph then $X_\H$ is a nearest neighbour shift of finite type. In addition it is also `isotropic' and `symmetric', that is, given vertices $a, b \in \H$ if $a$ is not allowed to sit next to $b$ in $X_\H$ for some coordinate direction, then $a$ is not allowed to sit next to $b$ in all coordinate directions. Most of the concepts related to shift spaces are introduced in Section \ref{section: Hom-shifts}.

Related to a shift space $X$ is its topological entropy denoted by $h_{top}(X)$ which measures the growth rate of the number of patterns allowed in $X$ with the size of the underlying shape (usually rectangular). For a given shift space, its computation is a very difficult task (look for instance in \cite{pavlovhardsquare2012} and the references within). We will focus on a different aspect: as in \cite{covensmital}, a shift space $X$ is called entropy minimal if for all shift spaces $Y\subsetneq X$, $h_{top}(Y)<h_{top}(X)$. Thus if a shift space has zero entropy and is entropy minimal then it is a topologically minimal system.

For $d=1$, it is well known that all irreducible nearest neighbour shifts of finite type are entropy minimal \cite{LM}. However not much is known about it in higher dimensions: Shift spaces with a strong mixing property called uniform filling are entropy minimal \cite{Schraudner2010minimal}, however nearest neighbour shifts of finite type with weaker mixing properties like block-gluing may not be entropy minimal \cite{boyle2010multidimensional}. If $\H$ is a four-cycle free graph then $X_\H$ is not even block-gluing (we do not prove this but is implied by our results). There has been some recent work \cite{lightwoodschraudnerentropy} which describes some conditions which are equivalent to entropy minimality for shifts of finite type. Our first main result is Theorem \ref{theorem:four cycle free entropy minimal} which states that $X_\H$ is entropy minimal for all connected four-cycle free graphs $\H$. Our approach to this seemingly combinatorial question will be via thermodynamic formalism, using measures of maximal entropy, more generally `adapted' Markov random fields. They are defined and described in Section \ref{section:thermodynamic formalism}.

Our second main result is with regard to the pivot property (Section \ref{section: the pivot property}). A shift space $X$ is said to have the pivot property if for all distinct configurations $x,y\in X$ which differ at finitely many sites, there exists a sequence $x=x^1, x^2, \ldots, x^n=y\in X$ such that successive configurations $x^i, x^{i+1}$ differ exactly at a single site. We will prove that for four-cycle free graphs $\H$, $X_{\H}$ has the pivot property (Theorem \ref{theorem: pivot property for four cycle free}). Many properties similar to the pivot property have appeared in the literature (often by the name local-move connectedness) \cite{brightwell2000gibbs,dominolocal2010,ordentlich2014data,Shiffieldribbon2002}. For instance consider the following problem: Let $G$ be a finite undirected graph without multiple edges and self-loops. Given two graph homomorphisms $x,y$ from $G$ to $H$, can we find graph homomorphisms $x=x^1, x^2, \ldots, x^n=y$ from $G$ to $H$ such that successive configurations $x^i, x^{i+1}$ differ exactly at a single site? Such a problem is called a homomorphism reconfiguration problem. If $\H$ is four-cycle free it was recently shown in \cite{Marcinfourcyclefree2014} that the homomorphism reconfiguration problem is solvable in time polynomial in the size of the graph $\H$. Our interest in such problems comes from the connections between the pivot property and the study of Markov and Gibbs cocycles \cite{chandgotia2013Markov}.

A critical part of our proofs for both Theorem \ref{theorem:four cycle free entropy minimal} and \ref{theorem: pivot property for four cycle free} depends on the identification of the associated height functions (defined in Section \ref{Section:heights}). Some of these ideas come from \cite{chandgotia2013Markov}: Let $C_n$ denote the cycle with vertices $0, 1, \ldots, n-1$. If $n \neq 1, 2, 4$ then it was proven that $X_{C_n}$ has the pivot property. Further, Lemma 6.7 in \cite{chandgotia2013Markov} implied that it is entropy minimal as well. We give a brief description of the latter: Given a configuration $x\in X_{C_n}$ it was proved that there exists a corresponding height function $h_x \in X_{\Z}$ such that $h_x\mod n= x$. (for $n=3$ also look at \cite{schmidt_cohomology_SFT_1995}) Further given any ergodic measure $\mu$ (assuming `adaptedness' cf. Section \ref{section:thermodynamic formalism}) on $X_{C_n}$ it was shown that a well-defined notion of slope (bounded by $-1$ and $1$) exists which measures the average rate of the increase of the height for every direction. If the slope is maximal in any direction it was proven that $\mu$ is frozen and otherwise it was shown that it is fully supported; frozen meaning that if $x, y\in supp(\mu)$ differ only at finitely many sites then $x=y$. From standard results in thermodynamic formalism it follows that $X_{C_n}$ is entropy minimal.

These ideas do not immediately translate to four-cycle free graphs. For this we will develop a different notion of `heights': Fix a connected four-cycle free graph $\H$. For all $u \in \H$ let $D_\H(u)$ denote the ball (for the graph distance) of radius $1$ around $u$. As in algebraic topology, $C$ is a covering space of $\H$ if there is a graph homomorphism (called the covering map) $f: C\longrightarrow \H$ such that for all $u \in \H$, $f^{-1}(D_{\H}(u))$ is a disjoint union of a constant number of subgraphs of $C$ isomorphic to $D_{\H}(u)$ via $f$. Given a graph $\H$, its universal cover (denoted by $E_\H$) is the unique covering space of $\H$ which is a tree. (Section \ref{section:universal covers} and \cite{Angluin80}) Denote by $\pi:E_\H\longrightarrow \H$ the corresponding covering map. $\pi$ induces a map from $X_{E_\H}$ to $X_\H$ (also denoted by $\pi$). It is not difficult to see that $\pi(X_{E_\H})\subset X_\H$; we will prove that the map is surjective and that the preimage of a configuration in $X_\H$ is unique once we fix the `lift' in $X_{E_\H}$ at any vertex of $\Z^d$. Thus for every $x\in X_\H$ we can associate $\t x \in X_{E_\H}$ such that $\pi(\t x)= x$ and a `height' function $h_x: \Z^d\times \Z^d \longrightarrow \Z$ such that $h_x(\m i , \m j)$ is the graph distance between $\t x_\mi$ and $\t x_\mj$. From here on the steps for the proofs of Theorems \ref{theorem:four cycle free entropy minimal} and \ref{theorem: pivot property for four cycle free} are similar to the steps for the corresponding proofs in \cite{chandgotia2013Markov}, but the proofs for the individual steps are quite different because unlike in \cite{chandgotia2013Markov} the height functions are not additive but subadditive, meaning
$$h_x(\mi,\mj)\leq h_x(\mi, \m k)+h_x(\m k, \mj).$$

To streamline the proofs, we use graph folding \cite{nowwinkler}; much of this is discussed in Section \ref{section:Folding, Entropy Minimality and the Pivot Property}.

\section{Shifts of Finite Type and Hom-Shifts} \label{section: Hom-shifts}
In this paper $\H$ will always denote an undirected graph without multiple edges and single isolated vertices. For such a graph we will denote the adjacency relation by $\sim_{\H}$ and the set of vertices of $\H$ by $\H$ (abusing notation). We identify $\Z^d$ with the set of vertices of the Cayley graph with respect to the standard generators $\m e_1,\m e_2, \ldots, \m e_d$, that is, $\m{i}\sim_{\Z^d}\m{j}$ if and only if $\|\m{i}-\m{j}\|_1=1$ where $\|\cdot\|_1$ is the $l^1$ norm. We drop the subscript in $\sim_{\H}$ when $\H=\Z^d$. Let $D_n$ and $B_n$ denote the $\Z^d$-balls of radius $n$ around $\m{0}$ in the $l^1$ and the $l^\infty$ norm respectively. The graph $C_n$ will denote the $n$-cycle where the set of vertices is $\{0, 1, 2, \ldots, n-1\}$ and $i\sim_{C_n} j$ if and only if $i \equiv j \pm 1\!\!\mod n$. The graph $K_n$ will denote the complete graph with $n$ vertices where the set of vertices is $\{1, 2, \ldots, n\}$ and $i \sim_{K_n} j$ if and only if $i \neq j$.

Let $\A$ be a finite \emph{alphabet} (with the discrete topology) and $\A^{\Z^d}$ be given the product topology, making it compact. We will refer to elements $x\in \A^{\Z^d}$ as \emph{configurations}, denoting them by $(x_{\m{i}})_{\m{i}\in \Z^d}$ where $x_{\m{i}}$ is the value of $x$ at $\m{i}$. For all $\m{i}\in \Z^d$ the map $\sigma^{\m{i}}:\A^{\Z^d}\longrightarrow \A^{\Z^d}$ given by
$$(\sigma^{\m{i}}(x))_{\m{j}}:=x_{\m{i}+\m{j}}$$
is called the \emph{shift map} and defines a $\Z^d$-action on $\A^{\Z^d}$. Closed subsets of $\A^{\Z^d}$ which are invariant under the shift maps are called \emph{shift spaces}. A \emph{sliding block code} from a shift space $X$ to a shift space $Y$ is a continuous map $f: X\longrightarrow Y$ which commutes with the shifts, that is, $\sigma^{\mi}\circ f= f\circ \sigma^\mi$ for all $\mi\in \Z^d$. A surjective sliding block code is called a \emph{factor map} and a bijective sliding block code is called a \emph{conjugacy}. We note that a conjugacy defines an equivalence relation; in fact, it has a continuous inverse since it is a continuous bijection between compact sets.

There is an alternate description of shift spaces using forbidden patterns: A \emph{pattern} is an element of $\A^F$ for some finite set $F\subset \Z^d$. Given a pattern $a\in \A^F$, we will often denote both the pattern and its corresponding cylinder set by $[a]_F$ or by $[x]_F$ when $x|_F=a$. For any set $F\subset \Z^d$ and $v\in \A$, $x|_F=v$ will mean that $x_\mi=v$ for all $\mi\in F$. For a set of (forbidden) patterns $\F$ define
$$X_\F:=\{x\in \A^{\Z^d}\:\Big{|}\: \sigma^{\m{i}}(x)|_F\notin \F \text{ for all }F\subset \Z^d\text{ and }\mi \in \Z^d\}.$$

It can be proved that a subset $X\subset \A^{\Z^d}$ is a shift space if and only if there exists a set of forbidden patterns $\F$ such that $X_\F= X$. Note that given two distinct sets of patterns $\F_1, \F_2$ it is possible that $X_{\F_1}=X_{\F_2}$. A \emph{shift of finite type} is a shift space $X$ such that $X=X_{\F}$ for some finite set $\F$. A \emph{nearest neighbour shift of finite type} is a shift of finite type $X$ such that $X=X_{\F}$ for some set $\F$ consisting of patterns on single edges and vertices of $\Z^d$. It follows from a simple recoding argument that any shift of finite type is conjugate to a nearest neighbour shift of finite type \cite{schimdt_fund_cocycle_98}. In this paper, we will focus on a special class of nearest neighbour shifts of finite type where the forbidden patterns are the same in every `direction':

Given a graph $\H$ let
$$X_{\H}:=\{x\in\H^{\Z^d}\:|\: x_{\m {i}}\sim_\H x_{\m{j}}\text{ for all }\m{i}\sim \m{j}\}.$$
Such spaces will be called \emph{hom-shifts}. Note that $X_\H$ is the space of graph homomorphisms from $\Z^d$ to $\H$. If $\H$ is finite and
$$\F_\H:=\{[v,w]_{\m{0},\m{e}_j}\:|\:v\nsim_\H w, 1\leq j\leq d\}$$
then $X_{\H}= X_{\F_{\H}}$. These are exactly the nearest neighbour shifts of finite type with symmetric and isotropic constraints. For example if the graph $\H$ is given by Figure \ref{Figure: Hard Square} then $X_\H$ is the hard square shift, that is, configurations with alphabet $\{0,1\}$ such that adjacent symbols cannot both be $1$. $X_{K_n}$ is the space of $n$-colourings of the graph, that is, configurations with alphabet $\{1,2, \ldots, n\}$ where all adjacent colours are distinct. We note that the properties, symmetry and isotropy, are not invariant under conjugacy.

$\F$ will always denote a set of patterns and $\H$ will always denote a graph, there will not be any ambiguity in the notations $X_\F, X_\H$.

A graph $\H$ is called \emph{four-cycle free} if it is finite, it has no self-loops and $C_4$ is not a subgraph of $\H$. For instance $K_4$ is not a four-cycle free graph.
\begin{figure}[h]
\centering
\includegraphics[angle=0,
width=.1\textwidth]{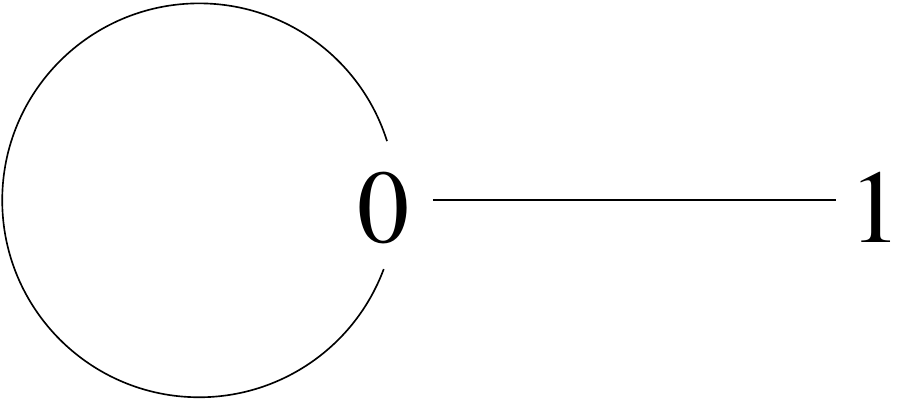}\caption{Graph for the Hard Square Shift}
\label{Figure: Hard Square}
\end{figure}

The set of \emph{globally allowed patterns} of a shift space $X$ on a set $A\subset \Z^d$ is
$$\L_A(X):=\{a\in \A^{A}\:\big{|}\:\text{ there exists }x\in X\text{ such that }x|_A=a\}.$$
Its \emph{language} is the set of all finite patterns appearing in $X$ that is
$$\L(X):=\bigcup_{A\subset \Z^d \text{ finite}}\L_A(X).$$
We comment that this is different from the set of locally allowed patterns: Let $X$ be a shift space with a forbidden list $\F$. Given a finite set $A$, a pattern $a\in \A^A$ is said to be \emph{locally allowed} if no pattern from $\F$ appears in $a$. In general it is undecidable for shifts of finite type whether a locally allowed pattern belongs to $\L(X)$ \cite{Robinson1971}; however it is decidable when $X$ is a hom-shift where it is sufficient to check whether the pattern extends to a locally allowed pattern on $B_n$ for some $n$.

The topological entropy of the shift space $X$ is the log growth rate of the number of allowed patterns in $X$, that is,
$$h_{top}(X):=\lim_{n\longrightarrow \infty}\frac{\log|\L_{B_n}(X)|}{|B_n|}.$$
The existence of the limit follows from subadditivity arguments via the well-known multivariate version of Fekete's Lemma \cite{silviofekete}. Moreover the topological entropy is an invariant under conjugacy (for $d=1$ look at Proposition 4.1.9 in \cite{LM}, the proof extends to higher dimensions). We remark that the computation of this invariant for shifts of finite type in $d>1$ is a hard problem and very little is known \cite{pavlovhardsquare2012}, however there are algorithms to compute approximating upper and lower bounds of the topological entropy of the hom-shifts \cite{symmtricfriedlan1997,louidor2010improved}. Further if $\H$ is a finite connected graph with at least two edges, then $h_{top}(X_\H)>0$:

\begin{prop}\label{proposition: hom-space positive entropy}
Let $\H$ be a finite graph with distinct vertices $a, b$ and $c$ such that $a\sim_\H b$ and $b\sim_\H c$. Then $h_{top}(X_{\H})\geq\frac{\log{2}}{2}$.
\end{prop}
\begin{proof}

It is sufficient to see this for a graph $\H$ with exactly three vertices $a$, $b$ and $c$ such that $a\sim_\H b$ and $b\sim_\H c$. For such a graph any configuration in $X_\H$ is composed of $b$ on one partite class of $\Z^d$ and a free choice between $a$ and $c$ for vertices on the other partite class. Then
$$|\L_{B_n}(X_\H)|=2^{{\lfloor\frac{(2n+1)^d}{2}\rfloor}}+2^{{\lceil\frac{(2n+1)^d}{2}\rceil}}$$
proving that $h_{top}(X_{\H})=\frac{\log{2}}{2}$.
\end{proof}

A shift space $X$ is called \emph{entropy minimal} if for all shift spaces $Y\subsetneq X$, $h_{top}(X)>h_{top}(Y)$. In other words, a shift space $X$ is entropy minimal if forbidding any word causes a drop in entropy. From \cite{quastrow2000} we know that every shift space contains an entropy minimal shift space with the same entropy and also a characterisation of same entropy factor maps on entropy minimal shifts of finite type.

One of the main results of this paper is the following:
\begin{thm}\label{theorem:four cycle free entropy minimal}
Let $\H$ be a connected four-cycle free graph. Then $X_\H$ is entropy minimal.
\end{thm}
For $d=1$ all irreducible shifts of finite type are entropy minimal \cite{LM}. A necessary condition for the entropy minimality of $X_\H$ is that $\H$ has to be connected.
\begin{prop}\label{proposition:entropy requires connectivity}
Suppose $\H$ is a finite graph with connected components $\H_1, \H_2, \ldots \H_r$. Then $h_{top}(X_\H)=\max_{1\leq i \leq r}h_{top}(X_{\H_i})$.
\end{prop}
This follows from the observation that
$$\max_{1\leq i \leq r}|\L_{B_n}({X_{\H_i}})|\leq |\L_{B_n}({X_\H})|= \sum_{i=1}^r|\L_{B_n}({X_{\H_i}})|\leq r\max_{1\leq i \leq r}|\L_{B_n}({X_{\H_i}})|.$$

\section{Thermodynamic Formalism}\label{section:thermodynamic formalism}
Here we give a brief introduction of thermodynamic formalism. For more details one can refer to \cite{Rue,walters-book}.

By $\mu$ we will always mean a shift-invariant Borel probability measure on a shift space $X$. The \emph{support} of $\mu$ denoted by $supp(\mu)$ is the intersection of all closed sets $Y \subset X$ for which $\mu(Y)= 1$. Note that $supp(\mu)$ is a shift space as well.
The \emph{measure theoretic entropy} is
\begin{equation*}
h_\mu:=\lim_{i \rightarrow \infty}\frac{1}{|D_i|}H^{D_i}_{\mu},
\end{equation*}
\noindent where $H^{D_i}_{\mu}$ is the Shannon-entropy of $\mu$ with respect to the partition of $X$ generated by the cylinder sets on $D_i$, the definition of which is given by:
\begin{equation*}
H^{D_i}_{\mu}:=\sum_{a\in \L_{D_i}(X)}-\mu([a]_{D_i})\log{\mu([a]_{D_i})},
\end{equation*} with the understanding that $0\log 0=0$.

A shift-invariant probability measure $\mu$ is a \emph{measure of maximal entropy} of $X$ if the maximum of $\nu \mapsto h_\nu$ over all shift-invariant probability measures on $X$ is obtained at $\mu$. The existence of measures of maximal entropy follows from upper-semi-continuity of the function $\nu \mapsto h_\nu$ with respect to the weak-$*$ topology.

Further the well-known \emph{variational principle} for topological entropy of $\Z^d$-actions asserts that if $\mu$ is a measure of maximal entropy $h_{top}(X)=h_\mu$ whenever $X$ is a $\Z^d$-shift space.

The following is a well-known characterisation of entropy minimality (it is used for instance in the proof of Theorem 4.1 in \cite{meestersteif2001}):
\begin{prop}
\label{proposition:entropyviamme}
A shift space $X$ is entropy minimal if and only if every measure of maximal entropy for $X$ is fully supported.
\end{prop}
We understand this by the following: Suppose $X$ is entropy minimal and $\mu$ is a measure of maximal entropy for $X$. Then by the variational principle for $X$ and $supp(\mu)$ we get
$$h_{top}(X)=h_\mu\leq h_{top}(supp(\mu))\leq h_{top}(X)$$
proving that $supp(\mu)=X$. To prove the converse, suppose for contradiction that $X$ is not entropy minimal and consider $Y\subsetneq X$ such that $h_{top}(X)= h_{top}(Y)$. Then by the variational principle there exists a measure $\mu$ on $Y$ such that $h_\mu= h_{top}(X)$. Thus $\mu$ is a measure of maximal entropy for $X$ which is not fully supported.

Further is known if $X$ is a nearest neighbour shift of finite type; this brings us to Markov random fields which we introduce next.

Given a set $A\subset \Z^d$ we denote the \emph{$r$-boundary} of $A$ by $\partial_r A$, that is,
$$\partial_r A=\{w\in \Z^d\setminus A\:\Big \vert\: \|w-v\|_1\leq r \text{ for some }v\in A\}.$$
The \emph{1-boundary} will be referred to as the \emph{boundary} and denoted by $\partial A$.
A \emph{Markov random field} on $\mathcal{A}^{\Z^d}$ is a Borel probability measure $\mu$ with the property that
for all finite $A, B \subset \Z^d$ such that $\partial A \subset B \subset A^{c}$ and $a \in \A^A, b \in \A^B$ satisfying $\mu([b]_B)>0$
\begin{equation*}
\mu([a]_A\;\Big\vert\;[b]_B)= \mu([a]_A\;\Big\vert\;[b]_{ \partial A}).
\end{equation*}

In general Markov random fields are defined over graphs much more general than $\Z^d$, however we restrict to the $\Z^d$ setting in this paper.

A \emph{uniform Markov random field} is a Markov random field $\mu$ such that further

\begin{equation*}
\mu([a]_A\;\Big\vert\;[b]_{ \partial A})=\frac{1}{n_{A,b|_{\partial A}}}
\end{equation*}
where $n_{A,b|_{\partial A}}=|\{a\in \A^A\:|\: \mu([a]_A\cap [b]_{\partial A})>0\}|$.

Following \cite{petersen_schmidt1997, schmidt_invaraint_cocycles_1997}, we denote by $\Delta_X$ the \emph{homoclinic equivalence relation} of a shift space $X$, which is given by
\begin{equation*}
\Delta_X := \{(x,y)\in X\times X\;|\; x_{\m i}=y_{\m i} \text{ for all but finitely many } \m i\in \Z^d\}.
\end{equation*}

We say that a measure $\mu$ is \emph{adapted} with respect to a shift space $X$ if
$supp(\mu)\subset X$ and
\begin{equation*}
x\in supp(\mu) \Longrightarrow \{y\in X\:|\: (x,y)\in \Delta_X\}\subset supp(\mu).
\end{equation*}

To illustrate this definition, let $X\subset \{0,1\}^{\Z}$ consist of configurations in $X$ in which at most a single $1$ appears. $X$ is uniquely ergodic; the delta-measure $\delta_{0^{\infty}}$ is the only shift-invariant measure on $X$. But
$$supp(\delta_{0^\infty})= \{0^\infty\}\subsetneq \{y\in X\;|\; 0^\infty_i= y_i \text{ for all but finitely many } i\in \Z\}=X,$$
proving that it is not adapted. On the other hand, since the homoclinic relation of $\A^{\Z^d}$ is minimal, meaning that for all $x\in \A^{\Z^d}$
$$\overline{\{y\in \A^{\Z^d}\;|\; y_{\m i}=x_{\m i} \text{ for all but finitely many } \mi\in \Z^d\}}= \A^{\Z^d},$$
it follows that a probability measure on $\A^{\Z^d}$ is adapted if and only if it is fully supported.

The relationship between measures of maximal entropy and Markov random fields is established by the following theorem. This is a special case of the Lanford-Ruelle theorem \cite{lanfruell,Rue}.

\begin{thm} All measures of maximal entropy on a nearest neighbour shift of finite type $X$ are shift-invariant uniform Markov random fields $\mu$ adapted to $X$.\label{thm:equiGibbs}
\end{thm}

The converse is also true under further mixing assumptions on the shift space $X$ (called the D-condition). The full strength of these statements is obtained by looking at \emph{equilibrium states} instead of measures of maximal entropy. The measures obtained there are not uniform Markov random fields, rather Markov random fields where the conditional probabilities are weighted via an \emph{interaction} giving rise to \emph{Gibbs states}. Uniform Markov random fields are Gibbs states with interaction zero.

We will often restrict our proofs to the ergodic case. We can do so via the following standard facts implied by Theorem $14.15$ in \cite{Georgii} and Theorem 4.3.7 in \cite{kellerequ1998}:

\begin{thm}\label{theorem: ergodic decomposition of markov random fields}
Let $\mu$ be a shift-invariant uniform Markov random field adapted to a shift space $X$. Let its ergodic decomposition be given by a measurable map $x\longrightarrow \mu_x$ on $X$, that is, $\mu= \int_X \mu_x d\mu$. Then $\mu$-almost everywhere the measures $\mu_x$ are shift-invariant uniform Markov random fields adapted to $X$ such that $supp(\mu_x)\subset supp(\mu)$. Moreover $\int h_{\mu_x} d\mu(x)= h_\mu$.
\end{thm}

We will prove the following:
\begin{thm}\label{theorem: MRF fully supported }
Let $\H$ be a connected four-cycle free graph. Then every ergodic probability measure adapted to $X_\H$ with positive entropy is fully supported.
\end{thm}

This implies Theorem \ref{theorem:four cycle free entropy minimal} by the following: The Lanford-Ruelle theorem implies that every measure of maximal entropy on $X_\H$ is a uniform shift-invariant Markov random field adapted to $X_\H$. By Proposition \ref{proposition: hom-space positive entropy} and the variational principle we know that these measures have positive entropy. By Theorems \ref{theorem: ergodic decomposition of markov random fields} and \ref{theorem: MRF fully supported } they are fully supported. Finally by Proposition \ref{proposition:entropyviamme}, $X_\H$ is entropy minimal.

Alternatively, the conclusion of Theorem \ref{theorem: MRF fully supported } can be obtained via some strong mixing conditions on the shift space; we will describe one such assumption. A shift space $X$ is called \emph{strongly irreducible} if there exists $g>0$ such that for all $x, y \in X$ and $A, B\subset \Z^d$ satisfying $\min_{\m i \in A, \m j \in B}\|\m i - \m j\|_1\geq g$, there exists $z\in X$ such that $z|_{A}= x|_A$ and $z|_B= y|_B$. For such a space, the homoclinic relation is minimal implying the conclusion of Theorem \ref{theorem: MRF fully supported } and further, that every probability measure adapted to $X$ is fully supported. Note that this does not prove that $X$ is entropy minimal unless we assume that $X$ is a nearest neighbour shift of finite type. Such an argument is used in the proof of Lemma 4.1 in \cite{meestersteif2001} which implies that every strongly irreducible shift of finite type is entropy minimal. A more combinatorial approach was used in \cite{Schraudner2010minimal} to show that general shift spaces with a weaker mixing property called uniform filling are entropy minimal.

\section{The Pivot Property}\label{section: the pivot property}
A \emph{pivot} in a shift space $X$ is a pair of configurations $(x,y)\in X$ such that $x$ and $y$ differ exactly at a single site. A subshift $X$ is said to have \emph{the pivot property} if for all distinct $(x,y)\in \Delta_X$ there exists a finite sequence of configurations $x^{(1)}=x, x^{(2)},\ldots, x^{(k)}=y \in X$ such that each $(x^{(i)}, x^{(i+1)})$ is a pivot. In this case we say $x^{(1)}=x, x^{(2)},\ldots, x^{(k)}=y$ is a \emph{chain of pivots} from $x$ to $y$.
Here are some examples of subshifts which have the pivot property:

\begin{enumerate}
\item Any subshift with a trivial homoclinic relation, that is, the homoclinic classes are singletons.
\item Any subshift with a safe symbol\footnote{A shift space $X\subset \A^{\Z^d}$ has a \emph{safe symbol} $\star$ if for all $x\in X$ and $A\subset \Z^d$ the configuration $z\in \A^{\Z^d}$ given by
\begin{equation*}
z_{\m i}:=\begin{cases}
x_{\m i} &\text{ if } \m i \in A\\
\star &\text{ if } \m i \in A^c
\end{cases}•
\end{equation*}
is also an element of $X$.}.
\item The hom-shifts $X_{C_r}$. This was proved for $r\neq 4$ in \cite{chandgotia2013Markov}, the result for $r=4$ is a special case of Proposition \ref{proposition: frozenfoldpivot}.
\item $r$-colorings of $\mathbb{Z}^d$ with $r\geq 2d+2$. (It is well-known, look for instance in Subsection 3.2 of \cite{chandgotia2013Markov})
\item\label{item: pivot property list number 5} $X_\H$ when $\H$ is dismantlable. \cite{brightwell2000gibbs}
\end{enumerate}
We generalise the class of examples given by (\ref{item: pivot property list number 5}) in Proposition \ref{proposition: frozenfoldpivot}. It is not true that all hom-shifts have the pivot property.
\begin{figure}[h]
\centering
\includegraphics[angle=0,
width=.1\textwidth]{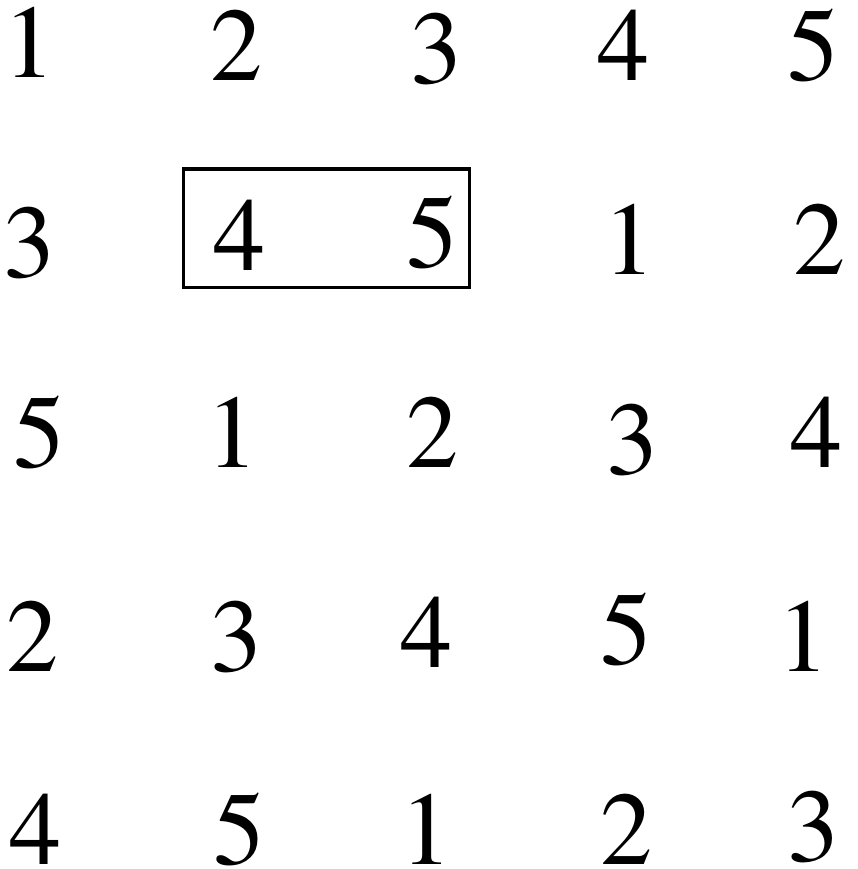}\caption{Frozen Pattern}
\label{Figure: Five colour}
\end{figure}
The following was observed by Brian Marcus: Recall that $K_n$ denotes the complete graph with $n$ vertices. $X_{K_4}, X_{K_5}$ do not possess the pivot property if the dimension is two. For instance consider a configuration in $X_{K_5}$ which is obtained by tiling the plane with the pattern given in Figure \ref{Figure: Five colour}. It is clear that the symbols in the box can be interchanged but no individual symbol can be changed. Therefore $X_{K_5}$ does not have the pivot property. However both $X_{K_4}$ and $X_{K_5}$ satisfy a more general property as discussed in Subsection \ref{subsection: Hom-shifts and the pivot property}.

The following theorem is another main result in this paper.
\begin{thm}\label{theorem: pivot property for four cycle free}
For all four-cycle free graphs $\H$, $X_\H$ has the pivot property.
\end{thm}
It is sufficient to prove this theorem for four-cycle free graphs $\H$ which are connected because of the following proposition:
\begin{prop}\label{proposition: pivot for disconnected}
Let $X_1, X_2, \ldots, X_n$ be shift spaces on disjoint alphabets such that each of them has the pivot property. Then $\cup_{i=1}^n X_i$ also has the pivot property.
\end{prop}
This is true since $(x, y)\in \Delta_{\cup_{i=1}^n X_i}$ implies $(x, y)\in \Delta_{X_i}$ for some $1\leq i \leq n$.

\section{Folding, Entropy Minimality and the Pivot Property}\label{section:Folding, Entropy Minimality and the Pivot Property} Given a graph $\H$ we say that a vertex $v$ \emph{folds} into a vertex $w$ if and only if $u \sim_\H v$ implies $u \sim_\H w$. In this case the graph $\H\setminus \{v\}$ is called a \emph{fold} of $\H$. The folding gives rise to a `retract' from $\H$ to $\H\setminus\{v\}$, namely the graph homomorphism from $\H$ to $\H\setminus \{v\}$ which is the identity on $\H\setminus \{v\}$ and sends $v$ to $w$. This was introduced in \cite{nowwinkler} to help characterise cop-win graphs and used in \cite{brightwell2000gibbs} to establish many properties which are preserved under `folding' and `unfolding'. Given a finite tree $\H$ with more than two vertices note that a leaf vertex (vertex of degree $1$) can always be folded to some other vertex of the tree. Thus starting with $\H$, there exists a sequence of folds resulting in a single edge. In fact using a similar argument we can prove the following proposition.

\begin{prop}\label{proposition:folding trees into other trees}
Let $\H\subset \H^\prime$ be trees. Then there is a graph homomorphism $f: \H^\prime \longrightarrow \H$ such that $f|_{\H}$ is the identity map.
\end{prop}

To show this, first note that if $\H\subsetneq\H^\prime$ then there is a leaf vertex in $\H^\prime$ which is not in $\H$. This leaf vertex can be folded into some other vertex in $\H^\prime$. Thus by induction on $|\H^\prime \setminus \H|$ we can prove that there is a sequence of folds from $\H^\prime$ to $\H$. Corresponding to this sequence of folds we obtain a graph homomorphism from $\H^\prime$ to $\H$ which is the identity on $\H$.

Here we consider a related notion for shift spaces. Given a nearest neighbour shift of finite type $X\subset \A^{\Z^d}$, \emph{the neighbourhood} of a symbol $v\in \A$ is given by
$$N_X(v):=\{a \in \A^{\partial \m 0}\:|\: [v]_{\m 0}\cap [a]_{\partial \m 0}\in \L_{D_1}(X)\},$$
that is the collection of all patterns which can `surround' $v$ in $X$. We will say that $v$ \emph {config-folds} into $w$ in $X$ if $N_X(v)\subset N_X(w)$. In such a case we say that $X$ \emph{config-folds} to $X\cap(\A\setminus \{v\})^{\Z^d}$. Note that $X\cap(\A\setminus \{v\})^{\Z^d}$ is obtained by forbidding $v$ from $X$ and hence it is also a nearest neighbour shift of finite type. Also if $X=X_\H$ for some graph $\H$ then $v$ config-folds into $w$ in $X_\H$ if and only if $v$ folds into $w$ in $\H$. Thus if $\H$ is a tree then there is a sequence of folds starting at $X_\H$ resulting in the two checkerboard configurations with two symbols (the vertices of the edge which $\H$ folds into). This property is weaker than the notion of folding introduced in \cite{chandgotiahammcliff2014}.

The main thrust of this property in our context is: if $v$ config-folds into $w$ in $X$ then given any $x\in X$, every appearance of $v$ in $x$ can be replaced by $w$ to obtain another configuration in $X$. This replacement defines a factor (surjective, continuous and shift-invariant) map $f: X\longrightarrow X\cap(\A\setminus \{v\})^{\Z^d}$ given by
\begin{equation*}
(f(x))_{\m i}:=\begin{cases}
x_{\m i}&\text{ if } x_{\m i}\neq v\\
w&\text{ if } x_{\m i}= v.
\end{cases}•
\end{equation*}
Note that the map $f$ defines a `retract' from $X$ to $X\cap(\A\setminus \{v\})^{\Z^d}$. Frequently we will config-fold more than one symbol at once (especially in Section \ref{section: Proof of the main theorems}):

Distinct symbols $v_1, v_2, \ldots, v_n$ \emph{config-fold disjointly} into $w_1, w_2, \ldots, w_n$ in $X$ if $v_i$ config-folds into $w_i$ and $v_i\neq w_j$ for all $1\leq i, j \leq n$. In this case the symbols $v_1, v_2, \ldots, v_n$ can be replaced by $w_1, w_2, \ldots, w_n$ simultaneously for all $x \in X$. Suppose $v_1, v_2,\ldots v_n$ is a maximal set of symbols which can be config-folded disjointly in $X$. Then $X\cap(\A\setminus \{v_1, v_2, \ldots, v_n\})^{\Z^d}$ is called a \emph{full config-fold} of $X$.

For example consider a tree $\H:=(\V,\E1)$ where $\V:=\{v_1, v_2, v_3, \ldots, v_{n+1}\}$ and $\E1:=\{(v_i, v_{n+1})\:|\: 1\leq i \leq n\}$. For all $1\leq i \leq n$, $\V\setminus \{v_i, v_{n+1}\}$ is a maximal set of symbols which config-folds disjointly in $X_\H$ resulting in the checkerboard patterns with the symbols $v_i$ and $v_{n+1}$ for all $1\leq i \leq n$. Thus the full config-fold of a shift space is not necessarily unique. However it is unique up to conjugacy:

\begin{prop}\label{Proposition: Uniqueness of full config-fold}
The full config-fold of a nearest neighbour shift of finite type is unique up to conjugacy via a change of the alphabet.
\end{prop}
The ideas for the following proof come essentially from the proof of Theorem 4.4 in \cite{brightwell2000gibbs} and discussions with Prof. Brian Marcus.
\begin{proof}
Let $X\subset \A^{\Z^d}$ be a nearest neighbour shift of finite type and
$$M:=\{v\in \A \:|\: \text{ for all }w\in \A,\  v \text{ config-folds into }w \Longrightarrow w \text{ config-folds into }v\}.$$
There is a natural equivalence relation $\equiv$ on $M$ given by $v\equiv w$ if $v$ and $w$ config-fold into each other. Let $A_1, A_2, A_3, \ldots, A_r\subset M$ be the corresponding partition. Clearly for all distinct $v, w\in M$, $v$ can be config-folded into $w$ if and only if $v, w\in A_i$ for some $i$. It follows that $A\subset A_i$ can be config-folded disjointly if and only if $\emptyset\neq A\neq A_i$.

Let $v\in \A\setminus M$. We will prove that $v$ config-folds to a symbol in $M$. By the definition of $M$ there exists $v_1\in \A$ such that $N_X(v)\subsetneq N_X(v_1)$. If $v_1\in M$ then we are done, otherwise choose $v_2\in \A$ such that $N_X(v_1)\subsetneq N_X(v_2)$. Continuing this process recursively we can find a sequence $v= v_0, v_1, v_2, \ldots, v_n$ such that $N_X(v_{i-1})\subsetneq N_X(v_i)$ for all $1\leq i \leq n$ and $v_n\in M$. Thus $v$ config-folds into $v_n$, a symbol in $M$. Further if $v$ config-folds to a symbol in $A_i$ it can config-fold to all the symbols in $A_i$. Therefore $B$ is a maximal subset of symbols in $\A$ which can be config-folded disjointly if and only if $B=\cup_{i=1}^rB_i\cup (\A\setminus M)$ where $B_i\subset A_i$ and $|A_i\setminus B_i|=1$. Let $B'\subset \A$ be another such maximal subset, $\A\setminus B:=\{b_1, b_2, \ldots, b_r\}$ and $\A\setminus B':=\{b'_1, b'_2, \ldots, b'_r\}$ where $b_i, b'_i\in A_i$. Then the map
$$f: X\cap (\A\setminus B)^{\Z^d} \longrightarrow X\cap (\A\setminus B')^{\Z^d} \text{ given by } f(x) := y\text{ where } y_{\m i}=b'_j \text{ whenever }x_{\m i}=b_j$$
is the required change of alphabet between the two full config-folds of $X$. \end{proof}

Let $X\cap(\A\setminus \{v_1, v_2, \ldots, v_n\})^{\Z^d}$ be a \emph{full config-fold} of $X$ where $v_i$ config-folds into $w_i$ for all $1\leq i\leq n$. Consider $f_X: \A \longrightarrow\A\setminus \{v_1, v_2, \ldots, v_n\}$ given by
\begin{equation*}
f_X(v):=\begin{cases}
v&\text{ if } v\neq v_j \text{ for all }1\leq j\leq n\\
w_j&\text{ if } v= v_j\text{ for some }1\leq j \leq n.
\end{cases}•
\end{equation*}
\noindent This defines a factor map $f_X: X\longrightarrow X\cap(\A\setminus \{v_1, v_2, \ldots, v_n\})^{\Z^d}$ given by $(f_X(x))_{\mi}:= f_X(x_\mi)$ for all $\mi \in \Z^d$. $f_X$ denotes both the factor map and the map on the alphabet; it should be clear from the context which function is being used.

In many cases we will fix a configuration on a set $A\subset \Z^d$ and apply a config-fold on the rest. Hence we define the map $f_{X,A}: X\longrightarrow X$ given by
\begin{equation*}
(f_{X,A}(x))_{\m i}:=\begin{cases}
x_{\m i}&\text{ if } \m i \in A\\
f_X(x_{\m i})&\text{ otherwise.}
\end{cases}•
\end{equation*}•

The map $f_{X,A}$ can be extended beyond $X$:

\begin{prop}\label{prop: folding_fixing_a_set}
Let $X\subset Y$ be nearest neighbour shifts of finite type, $Z$ be a full config-fold of $X$ and $y\in Y$ such that for some $A\subset \Z^d$, $y|_{A^c\cup\partial (A^c)}\in \L_{A^c\cup\partial (A^c)}(X)$. Then the configuration $z$ given by
\begin{equation*}
z_{\m i}:=\begin{cases}
y_{\m i}&\text{ if } \m i \in A\\
f_X(y_{\m i})&\text{ otherwise}
\end{cases}•
\end{equation*}
is an element of $Y$. Moreover $z|_{A^c}\in \L_{A^c}(Z)$.
\end{prop}
Abusing the notation, in such cases we shall denote the configuration $z$ by $f_{X, A}(y)$.

If $A^c$ is finite, then $f_{X,A}$ changes only finitely many coordinates. These changes can be applied one by one, that is, there is a chain of pivots in $Y$ from $y$ to $f_{X,A}(y)$.

A nearest neighbour shift of finite type which cannot be config-folded is called a \emph{stiff shift}. We know from Theorem 4.4 in \cite{brightwell2000gibbs} that all the stiff graphs obtained by a sequence of folds of a given graph are isomorphic. By Proposition \ref{Proposition: Uniqueness of full config-fold} the corresponding result for nearest neighbour shifts of finite type immediately follows:
\begin{prop}\label{proposition:uniqueness of stiff shifts}
The stiff shift obtained by a sequence of config-folds starting with a nearest neighbour shift of finite type is unique up to conjugacy via a change of the alphabet.
\end{prop}

Starting with a nearest neighbour shift of finite type $X$ the \emph{fold-radius} of $X$ is the smallest number of full config-folds required to obtain a stiff shift. If $\H$ is a tree then the fold-radius of $X_\H$ is equal to
$$\left\lfloor\frac{diameter(\H)}{2}\right\rfloor.$$
Thus for every nearest neighbour shift of finite type $X$ there is a sequence of full config-folds (not necessarily unique) which starts at $X$ and ends at a stiff shift of finite type. Let the fold-radius of $X$ be $r$ and $X= X_0, X_1, X_2, \ldots, X_r$ be a sequence of full config-folds where $X_r$ is stiff. This generates a sequence of maps $f_{X_i}:X_{i}\longrightarrow X_{i+1}$ for all $0\leq i \leq r-1$. In many cases we will fix a pattern on $D_n$ or $D_n^c$ and apply these maps on the rest of the configuration. Consider the maps $I_{X,n}:X\longrightarrow X$ and $O_{X,n}:X\longrightarrow X$ (for $n>r$) given by
\begin{equation*}
I_{X,n}(x):=f_{X_{r-1},D_{n+r-1} }\left(f_{X_{r-2}, D_{n+r-2}}\left(\ldots\left(f_{X_{0}, D_n}(x)\right)\ldots\right)\right)\text{(Inward Fixing Map)}
\end{equation*}•
and
\begin{eqnarray*}
O_{X,n}(x):=f_{X_{r-1}, D_{n-r+1}^c }\left(f_{X_{r-2}, D_{n-r+2}^c}\left(\ldots\left(f_{X_{0}, D_n^c}(x)\right)\ldots\right)\right)\text{(Outward Fixing Map)}.
\end{eqnarray*}•
Similarly we consider maps which do not fix anything, $F_X: X\longrightarrow X_r$ given by
\begin{eqnarray*}
F_X(x):= f_{X_{r-1}}\left(f_{X_{r-2}}\left(\ldots\left(f_{X_{0}}(x)\right)\ldots\right)\right).
\end{eqnarray*}
Note that $D_k\cup \partial D_k= D_{k+1}$ and $D_k^c\cup \partial (D_k^c)=D_{k-1}^c$. This along with repeated application of Proposition \ref{prop: folding_fixing_a_set} implies that the image of $I_{X,n}$ and $O_{X,n}$ lie in $X$. This also implies the following proposition:

\begin{prop}[The Onion Peeling Proposition]\label{prop: folding_ to _ stiffness_fixing_a_set}
Let $X\subset Y$ be nearest neighbour shifts of finite type, $r$ be the fold-radius of $X$, $Z$ be a stiff shift obtained by a sequence of config-folds starting with $X$ and $y^1, y^2\in Y$ such that $y^1|_{D_{n-1}^c}\in \L_{D_{n-1}^c}(X)$ and $y^2|_{D_{n+1}}\in \L_{D_{n+1}}(X)$. Let $z^1, z^2\in Y$ be given by
\begin{eqnarray*}
z^1&:=&f_{X_{r-1},D_{n+r-1} }\left(f_{X_{r-2}, D_{n+r-2}}\left(\ldots\left(f_{X_{0}, D_n}(y^1)\right)\ldots\right)\right)\\
z^2&:=&f_{X_{r-1}, D_{n-r+1}^c }\left(f_{X_{r-2},D_{n-r+2}^c}\left(\ldots\left(f_{X_{0}, D_n^c}(y^2)\right)\ldots\right)\right)\text{ for }n>r.
\end{eqnarray*}•
The patterns $z^1|_{D_{n+r-1}^c}\in \L_{D_{n+r-1}^c}(Z)$ and $z^2|_{D_{n-r+1}}\in \L_{D_{n-r+1}}(Z)$. If $y^1, y^2\in X$ then in addition
\begin{eqnarray*}
z^1|_{D_{n+r-1}^c}&=&F_X(y^1)|_{D_{n+r-1}^c}\text{ and}\\
z^2|_{D_{n-r+1}}&=&F_X(y^2)|_{D_{n-r+1}}.
\end{eqnarray*}•
\end{prop}

Abusing the notation, in such cases we shall denote the configurations $z^1$ and $z^2$ by $I_{X, n}(y^1)$ and $O_{X,n}(y^2)$ respectively. Note that $I_{X, n}(y^1)|_{D_n}= y^1|_{D_n}$ and $O_{X,n}(y^2)|_{D_n^c}= y^2|_{D_n^c}$. Also, $O_{X,n}$ is a composition of maps of the form $f_{X,A}$ where $A^c$ is finite; there is a chain of pivots in $Y$ from $y$ to $O_{X,n}(y)$.
There are two kinds of stiff shifts which will be of interest to us: A configuration $x\in \A^{\Z^d}$ is called \emph{periodic} if there exists $n \in \N$ such that $\sigma^{n \m e_{i}}(x)=x$ for all $1\leq i \leq d$. A configuration $x\in X$ is called \emph{frozen} if its homoclinic class is a singleton. This notion coincides with the notion of frozen coloring in \cite{brightwell2000gibbs}. A subshift $X$ will be called \emph{frozen} if it consists of frozen configurations, equivalently $\Delta_X$ is the diagonal. A measure on $X$ will be called \emph{frozen} if its support is frozen. Note that any shift space consisting just of periodic configurations is frozen. All frozen nearest neighbour shifts of finite type are stiff: Suppose $X$ is a nearest neighbour shift of finite type which is not stiff. Then there is a symbol $v$ which can be config-folded to a symbol $w$. This means that any appearance of $v$ in a configuration $x\in X$ can be replaced by $w$. Hence the homoclinic class of $x$ is not a singleton. Therefore $X$ is not frozen.

\begin{prop}\label{proposition: periodicfoldentropy} Let $X$ be a nearest neighbour shift of finite type such that a sequence of config-folds starting from $X$ results in the orbit of a periodic configuration. Then every shift-invariant probability measure adapted to $X$ is fully supported.
\end{prop}
\begin{prop}\label{proposition: frozenfoldpivot}
Let $X$ be a nearest neighbour shift of finite type such that a sequence of config-folds starting from $X$ results in a frozen shift. Then $X$ has the pivot property.
\end{prop}

\noindent\textbf{Examples:}
\begin{enumerate}
\item
$X:=\{0\}^{\Z^d}\cup \{1\}^{\Z^d}$ is a frozen shift space but not the orbit of a periodic configuration. Clearly the delta measure $\delta_{\{0\}^{\Z^d}}$ is a shift-invariant probability measure adapted to $X$ but not fully supported. A more non-trivial example of a nearest neighbour shift of finite type which is frozen but not the orbit of a periodic configuration is the set of the Robinson tilings $Y$ \cite{Robinson1971}. There are configurations in $Y$ which have the so-called ``fault lines''; they can occur at most once in a given configuration. Consequently for all shift-invariant probability measures on $Y$, the probability of seeing a fault line is zero. Thus no shift-invariant probability measure (and hence no adapted shift-invariant probability measure) on $Y$ is fully supported.

\item\label{Example: Safe Symbol}
Let $X$ be a shift space with a safe symbol $\star$. Then any symbol in $X$ can be config-folded into the safe symbol. By config-folding the symbols one by one, we obtain a fixed point $\{\star\}^{\Z^d}$. Thus any nearest neighbour shift of finite type with a safe symbol satisfies the hypothesis of both the propositions.

\item \label{Example: Folds to an edge}Suppose $\H$ is a graph which folds into a single edge (denoted by $Edge$) or a single vertex $v$ with a loop. Then the shift space $X_\H$ can be
config-folded to $X_{Edge}$ (which consists of two periodic configurations) or the fixed point $\{v\}^{\Z^d}$ respectively. In the latter case, the graph $\H$ is called \emph{dismantlable} \cite{nowwinkler}. Note that finite non-trivial trees and the graph $C_4$ fold into an edge. For dismantlable graphs $\H$, Theorem 4.1 in \cite{brightwell2000gibbs} implies the conclusions of Propositions \ref{proposition: periodicfoldentropy} and \ref{proposition: frozenfoldpivot} for $X_\H$ as well.
\end{enumerate}•

\begin{proof}[Proof of Proposition \ref{proposition: periodicfoldentropy}] Let $\mu$ be a shift-invariant probability measure adapted to $X$. To prove that $supp(\mu)= X$ it is sufficient to prove that for all $n\in \N$ and $x\in X$ that $\mu([x]_{D_n})>0$. Let $X_0=X$, $X_1$, $X_2$$,\ldots,$ $X_r$ be a sequence of full config-folds where $X_r:=\{ \sigma^{\m i_1}(p), \sigma^{\m i_2}(p),\ldots, \sigma^{\m i_{k-1}}(p) \}$ is the orbit of a periodic point. For any two configurations $z,w\in X$ there exists $\m i\in \Z^d$ such that $F_X(z)= F_X(\sigma^{\m i}(w)).$
Since $\mu$ is shift-invariant we can choose $y \in supp (\mu)$ such that $F_X(x)= F_X(y).$
Consider the configurations $I_{X,n}(x)$ and $O_{X,n+2r-1}(y)$. By Proposition \ref{prop: folding_ to _ stiffness_fixing_a_set} they satisfy the equations
\begin{eqnarray*}
I_{X,n}(x)|_{D_{n+r-1}^c}&=&F_X(x)|_{D_{n+r-1}^c}\text{ and }\\
O_{X,n+2r-1}(y)|_{D_{n+r}}&=&F_X(y)|_{D_{n+r}}.
\end{eqnarray*}
\noindent Then $I_{X,n}(x)|_{\partial D_{n+r-1}}= O_{X,n+2r-1}(y)|_{\partial D_{n+r-1}}$. Since $X$ is a nearest neighbour shift of finite type, the configuration $z$ given by
\begin{eqnarray*}
z|_{D_{n+r}}&:=&I_{X,n}(x)|_{D_{n+r}}\\
z|_{D_{n+r-1}^c}&:=&O_{X,n+2r-1}(y)|_{D_{n+r-1}^c}
\end{eqnarray*}
\noindent is an element of $X$. Moreover
\begin{eqnarray*}
z|_{D_{n}}&=&I_{X,n}(x)|_{D_{n}}=x|_{D_n}\\
z|_{D_{n+2r-1}^c}&=&O_{X,n+2r-1}(y)|_{D_{n+2r-1}^c}=y|_{D^c_{n+2r-1}}.
\end{eqnarray*}•
Thus $(y, z)\in \Delta_X$. Since $\mu$ is adapted we get that $z\in supp(\mu)$. Finally
$$\mu([x]_{D_n})=\mu([z]_{D_n})>0.$$
\end{proof}

Note that all the maps being discussed here, $f_X$, $f_{X,A}$, $F_X$, $I_{X,n}$ and $O_{X,n}$ are (not necessarily shift-invariant) single block maps, that is, maps $f$ where $\left(f(x)\right)_{\m i}$ depends only on $x_{\m i}$. Thus if $f$ is one such map and $x|_A= y|_A$ for some set $A\subset \Z^d$ then $f(x)|_A=f(y)|_A$; they map homoclinic pairs to homoclinic pairs.

\begin{proof}[Proof of Proposition \ref{proposition: frozenfoldpivot}] Let $X_0=X$, $X_1$, $X_2$$,\ldots,$ $X_r$ be a sequence of full config-folds where $X_r$ is frozen. Let $(x, y) \in \Delta_X$. Since $X_r$ is frozen, $F_{X}(x)= F_X(y)$. Suppose $x|_{D_n^c}= y|_{D_n^c}$ for some $n\in \N$. Then $O_{X,n+r-1}(x)|_{D_n^c}=O_{X,n+r-1}(y)|_{D_n^c}$. Also by Proposition \ref{prop: folding_ to _ stiffness_fixing_a_set},
$$O_{X, n+r-1}(x)|_{D_n}=F_X(x)|_{D_n}=F_X(y)|_{D_n}= O_{X, n+r-1}(y)|_{D_n}.$$
This proves that $O_{X,n+r-1}(x)=O_{X,n+r-1}(y)$. In fact it completes the proof since for all $z\in X$ there exists a chain of pivots in $X$ from $z$ to $O_{X,n+r-1}(z)$.
\end{proof}

\section{Universal Covers}\label{section:universal covers}
Most cases will not be as simple as in the proof of Propositions \ref{proposition: periodicfoldentropy} and \ref{proposition: frozenfoldpivot}. We wish to prove the conclusions of these propositions for hom-shifts $X_\H$ when $\H$ is a connected four-cycle free graph. Many ideas carry over from the proofs of these results because of the relationship of such graphs with their universal covers; we describe this relationship next. The results in this section are not original; look for instance in \cite{Stallingsgraph1983}. We mention them for completeness.

Let $\H$ be a finite connected graph with no self-loops. We denote by $d_\H$ the ordinary graph distance on $\H$ and by $D_\H(u)$, the \emph{ball of radius 1} around $u$. A graph homomorphism $\pi:\C\longrightarrow \H$ is called a \emph{covering map} if for some $n \in \N \cup \{\infty\}$ and all $u \in \H$, there exist disjoint sets $\{C_i\}_{i=1}^n\subset \C$ such that $\pi^{-1}\left(D_\H(u)\right)= \cup_{i=1}^n C_i $ and $\pi|_{C_i}: C_i\longrightarrow D_\H(u)$ is an isomorphism of the induced subgraphs for $1\leq i\leq n$. A \emph{covering space} of a graph $\H$ is a graph $\C$ such that there exists a covering map $\pi: \C\longrightarrow \H$.

A \emph{universal covering space} of $\H$ is a covering space of $\H$ which is a tree. Unique up to graph isomorphism \cite{Stallingsgraph1983}, these covers can be described in multiple ways. Their standard construction uses non-backtracking walks \cite{Angluin80}: A \emph{walk} on $\H$ is a sequence of vertices $(v_1, v_2, \ldots, v_n)$ such that $v_i\sim_\H v_{i+1}$ for all $1\leq i \leq n-1$. The \emph{length} of a walk $p=(v_1, v_2, \ldots, v_n)$ is $|p|=n-1$, the number of edges traversed on that walk. It is called \emph{non-backtracking} if $v_{i-1}\neq v_{i+1}$ for all $2\leq i \leq n-1$, that is, successive steps do not traverse the same edge. Choose a vertex $u \in \H$. The vertex set of the universal cover is the set of all non-backtracking walks on $\H$ starting from $u$; there is an edge between two such walks if one extends the other by a single step. The choice of the starting vertex $u$ is arbitrary; choosing a different vertex gives rise to an isomorphic graph. We denote the universal cover by $E_\H$. The covering map $\pi: E_\H\longrightarrow \H$ maps a walk to its terminal vertex. Usually, we will denote by $\t u, \t v$ and $\t w$ vertices of $E_\H$ such that $\pi(\t u)= u$, $\pi(\t v)= v$ and $\pi(\t w)= w$.

This construction shows that the universal cover of a graph is finite if and only if it is a finite tree. To see this if the graph has a cycle then the finite segments of the walk looping around the cycle give us infinitely many vertices for the universal cover. If the graph is a finite tree, then all walks must terminate at the leaves and their length is bounded by the diameter of the tree. In fact, the universal cover of a tree is itself while the universal cover of a cycle (for instance $C_4$) is $\Z$ obtained by finite segments of the walks $(1, 2, 3, 0, 1, 2, 3, 0, \ldots )$ and $(1, 0, 3, 2, 1, 0, 3, 2, \ldots )$.

Following the ideas of homotopies in algebraic topology, there is a natural operation on the set of walks: two walks can be joined together if one begins where the other one ends. More formally, given two walks $p=(v_1, v_2, \ldots, v_n)$ and $q=(w_1, w_2, \ldots, w_m)$ where $v_n=w_1$, consider $p\star q=(v_1, v_2, \ldots, v_n, w_2, w_3, \ldots, w_m)$. However even when $p$ and $q$ are non-backtracking $p\star q$ need not be non-backtracking. So we consider the walk $[p\star q]$ instead which erases the backtracking segments of $p \star q$, that is, if for some $i>1$, $v_{n-i+1}\neq w_{i}$ and $v_{n-j+1}=w_j$ for all $1\leq j \leq i-1$ then
$$[p\star q]:=(v_1, v_2, \ldots, v_{n-i+1}, w_{i-1}, w_{i}, \ldots, w_m).$$

This operation of erasing the backtracking segments is called \emph{reduction}, look for instance in \cite{Stallingsgraph1983}.
The following proposition is well-known (Section 4 of \cite{Stallingsgraph1983}) and shall be useful in our context as well:
\begin{prop}\label{proposition:isomorphism_of_universal_covering_space}
Let $\H$ be a finite connected graph without any self-loops. Then for all $\t{v}, \t w \in E_\H$ satisfying $\pi(\t v)= \pi(\t w)$ there exists a graph isomorphism $\phi: E_\H\longrightarrow E_\H$ such that $\phi(\t v)= \t w$ and $\pi \circ \phi = \pi$.
\end{prop}

To see how to construct this isomorphism, consider as an example $ (u)$, the empty walk on $\H$ and $(v_1, v_2, \ldots, v_n)$, some non-backtracking walk such that $v_1=v_n=u$. Then the map $\phi: E_\H\longrightarrow E_\H$ given by
$$\phi(\t w):= [(v_1, v_2, \ldots, v_n) \star \t w].$$
is a graph isomorphism which maps $(u)$ to $(v_1, v_2, \ldots, v_n)$; its inverse is $\psi: E_\H\longrightarrow E_\H$ given by
$$\psi(\t w):= [(v_n, v_{n-1}, \ldots, v_1)\star \t w].$$

The maps $\phi, \pi$ described above give rise to natural maps, also denoted by $\phi$ and $\pi$ where $$\phi:X_{E_\H}\longrightarrow X_{E_\H}$$
is given by $\phi(\t x)_\mi := \phi(\t x_\mi)$ and
$$\pi: X_{E_\H} \longrightarrow X_{\H}$$
is given by $\pi(\t x)_\mi:=\pi(\t x_\mi)$ for all $\mi \in \Z^d$ respectively. A \emph{lift} of a configuration $x\in X_\H$ is a configuration $\t{x}\in X_{E_\H}$ such that $\pi \circ \t{x}= x$.

Now we shall analyse some consequences of this formalism in our context. More general statements (where $\Z^d$ is replaced by a different graph) are true (under a different hypothesis on $\H$), but we restrict to the four-cycle free condition. We noticed in Section \ref{section:Folding, Entropy Minimality and the Pivot Property} that if $\H$ is a tree then $X_{\H}$ satisfies the conclusions of Theorems \ref{theorem: MRF fully supported } and \ref{theorem: pivot property for four cycle free}. Now we will draw a connection between the four-cycle free condition on $\H$ and the formalism in Section \ref{section:Folding, Entropy Minimality and the Pivot Property}.

\begin{prop}[Existence of Lifts]\label{proposition:covering_space_lifting}
Let $\H$ be a connected four-cycle free graph. For all $x\in X_\H$ there exists $\tilde{x}\in X_{E_\H}$ such that $\pi(\tilde{x})=x$. Moreover the lift $\tilde{x}$ is unique up to a choice of $\tilde{x}_{\m 0}$.
\end{prop}

\begin{proof}
We will begin by constructing a sequence of graph homomorphisms $\tilde{x}^n:D_n \longrightarrow E_\H$ such that $\pi \circ\tilde{x}^n =x|_{D_n}$ and $\tilde{x}^m|_{D_n}= \tilde{x}^n$ for all $m>n$. Then by taking the limit of these graph homomorphisms we obtain a graph homomorphism $\tilde{x}\in X_{E_\H}$ such that $\pi \circ\tilde{x}=x$. It will follow that given $\tilde{x}^0$ the sequence $\tilde{x}^n$ is completely determined proving that the lifting is unique up to a choice of $\tilde{x}_{\m {0}}$.

The recursion is the following: Let $\tilde{x}^n: D_n \longrightarrow E_\H$ be a given graph homomorphism for some $n\in \N\cup \{0\}$ such that $\pi \circ\tilde{x}^n=x|_{D_n}$. For any ${\m{ i}}\in D_{n+1}\setminus D_n$, choose a vertex ${\m{j}}\in D_n$ such that $\m{j}\sim \m{i}$. Then $\pi(\tilde{x}^n_{\m{j}})=x_{\m{j}}\sim x_{\m{i}}$. Since $\pi$ defines a local isomorphism between $E_\H$ and $\H$, there exists a unique vertex $\t v_{\m{i}}\sim \tilde{x}^n_{\m{j}} \in E_\H$ such that $\pi(\t v_{\m{i}})= x_{\m{i}}$. Define $\tilde{x}^{n+1}: D_{n+1}\longrightarrow E_\H$ by
\begin{equation*}
\tilde{x}^{n+1}_{\m{i}}:=\begin{cases}\tilde{x}^{n}_{\m{i}} &\text{if } \m{i}\in D_n\\\t v_{\m{i}} & \text{if } \m{i}\in D_{n+1}\setminus D_n.\end{cases}
\end{equation*}•
Then clearly $\pi \circ \tilde{x}^{n+1}= x|_{D_{n+1}}$ and $\tilde{x}^{n+1}|_{D_n}=\tilde{x}^n$. Note that the extension $\tilde{x}^{n+1}$ is uniquely defined given $\tilde{x}^n$.
We need to prove that this defines a valid graph homomorphism from $D_{n+1}$ to $E_\H$. Let $\m{i}\in D_{n+1}\setminus D_n$ and $\m{j}\in D_n$ be chosen as described above. Consider if possible any $\m{j}^\prime\neq \m{j} \in D_n$ such that $\m{j}^\prime \sim \m{i}$. To prove that $\tilde{x}^{n+1}$ is a graph homomorphism we need to verify that $\tilde{x}^{n+1}_{\m{j}^\prime}\sim \tilde{x}^{n+1}_{\m{i}}$.

Consider $\m{i}^\prime\in D_n$ such that $\m{i}^\prime\sim \m{j}$ and $\m{j}^\prime$. Then $\m{i}^\prime, \m{j}, \m{i}$ and $\m{j}^\prime$ form a four-cycle. Since $\H$ is four-cycle free either $x_{\m{i}^\prime}=x_{\m{i}}$ or $x_{\m{j}^\prime}= x_{\m{j}}$.

Suppose $x_{\m{i}^\prime}=x_{\m{i}}$; the other case is similar. Since $\pi$ is a local isomorphism and $\tilde{x}^{n+1}_{\m{i}},\tilde{x}^{n+1}_{\m{i^\p}} \sim \tilde{x}^{n+1}_{\m{j}}$, we get that $\tilde{x}^{n+1}_{\m{i}}=\tilde{x}^{n+1}_{\m{i}^\prime}$. But $\mi', \mj' \in D_n$ and $\t x^{n+1}|_{D_n}= \t x^{n}$ is a graph homomorphism; therefore $\tilde{x}^{n+1}_{\m{i}}=\tilde{x}^{n+1}_{\m{i}^\prime}\sim \tilde{x}^{n+1}_{\m{j}^\prime}$.
\end{proof}

\begin{corollary}\label{corollary:covering_space_lifting_homoclinic}
Let $\H$ be a connected four-cycle free graph and $x, y\in X_\H$. Consider some lifts $\tilde{x}, \tilde{y} \in X_{E_\H}$ such that $\pi(\tilde{x})= x $ and $\pi(\tilde{y})=y$. If for some $\m{i}_0 \in \Z^d$, $\tilde{x}_{\m{i}_0}= \tilde{y}_{\m{i}_0}$ then $\tilde{x}= \tilde{y}$ on the connected subset of
$$\{\m{j} \in \Z^d\:|\: x_{\m{j}}= y_{\m{j}}\}$$ which contains $\m{i}_0$.
\end{corollary}

\begin{proof}
Let $D$ be the connected component of $\{\m{i} \in \Z^d \:|\: x_{\m{i}} = y_{\m{i}}\}$ and $\t D$ be the connected component of $\{\m{i} \in \Z^d \:|\: \t x_{\m{i}} = \t y_{\m{i}}\}$ which contain $\m{i}_0$.

Clearly $\t D \subset D$. Suppose $\t D\neq D$. Since both $D$ and $\t D$ are non-empty, connected sets there exist $\m{i} \in D \setminus \t D$ and $\m{j} \in \t D$ such that $\m{i} \sim \m{j}$. Then $x_{\m{i}}= y_{\m{i}}$, $x_{\m{j}}= y_{\m{j}}$ and $\t x_{\m{j}}= \t y_{\m{j}}$. Since $\pi $ is a local isomorphism, the lift must satisfy $\t x_{\m{i}} = \t y_{\m{i}}$ implying $\m{i} \in \t D$. This proves that $D= \t D$.
\end{proof}
The following corollary says that any two lifts of the same graph homomorphism are `identical'.
\begin{corollary}\label{corollary:lift_are_isomorphic}
Let $\H$ be a connected four-cycle free graph. Then for all $\t x^1,\t x^2 \in X_{E_\H}$ satisfying $\pi(\t x^1) = \pi(\t x^2)= x$ there exists an isomorphism $\phi: E_\H \longrightarrow E_\H$ such that $\phi\circ \t x^1= \t x^2$.
\end{corollary}
\begin{proof} By Proposition \ref{proposition:isomorphism_of_universal_covering_space} there exists an isomorphism
$\phi: E_\H\longrightarrow E_\H$ such that $\phi(\t x^1_{\m{0}})= \t x^2_{\m{0}}$ and $\pi \circ \phi = \pi$. Then $(\phi\circ\t x^1)_{\m{0}} = \t x^2_{\m{0}}$ and $\pi (\phi\circ\t x^1)= (\pi \circ \phi)(\t x^1)= \pi (\t x^1)= x$. By Proposition \ref{proposition:covering_space_lifting} $\phi\circ\t x^1= \t x^2$.
\end{proof}

It is worth noting at this point the relationship of the universal cover described here with the universal cover in algebraic topology. Undirected graphs can be identified with $1$ dimensional CW-complexes where the set of vertices correspond to the $0$-cells, the edges to the $1$-cells of the complex and the attaching map sends the end-points of the edges to their respective vertices. With this correspondence in mind the (topological) universal covering space coincides with the (combinatorial) universal covering space described above; indeed a $1$ dimensional CW-complex is simply connected if and only if it does not have any loops, that is, the corresponding graph does not have any cycles; it is a tree. The uniqueness, existence and many such facts about the universal covering space follow from purely topological arguments; for instance look in Chapter $13$ in \cite{MunkresTopology75} or Chapters $5$ and $6$ in \cite{Masseyanintroduction1977}.

\section{Height Functions and Sub-Cocycles}\label{Section:heights}
Existence of lifts as described in the previous section enables us to measure the `rigidity' of configurations. In this section we define height functions and subsequently the slope of configurations, where steepness corresponds to this `rigidity'. The general method of using height
functions is usually attributed to J.H.Conway \cite{ThurstontilinggroupAMM}.

Fix a connected four-cycle free graph $\H$. Given $x\in X_\H$ we can define the corresponding \emph{height function} $h_x:\Z^d\times \Z^d\longrightarrow \Z$ given by $h_x({\m{i}},{\m{j}}):=d_{E_\H}(\tilde{x}_{\m{i}},\tilde{x}_{\m{j}} )$ where $\tilde{x}$ is a lift of $x$. It follows from Corollary \ref{corollary:lift_are_isomorphic} that $h_x$ is independent of the lift $\tilde{x}$.

Given a finite subset $A\subset \Z^d$ and $x\in X_\H$ we define the \emph{range of $x$ on A} as
\begin{equation*}
Range_A(x):=\max_{\mj_1, \mj_2\in A} h_x(\mj_1, \mj_2).
\end{equation*}
For all $x\in X_\H$
\begin{equation*}
Range_A(x)\leq Diameter(A)
\end{equation*}
and more specifically
\begin{equation} \label{equation:diameter_bounds_height}
Range_{D_n}(x)\leq 2n
\end{equation}
for all $n \in \N$. Since $\t x\in X_{E_\H}$ is a map between bipartite graphs it preserves the parity of the distance function, that is, if $\m i, \m j \in \Z^d$ and $x\in X_\H$ then the parity of $\|\m i - \m j\|_1$ is the same as that of $h_x(\m i, \m j)$. As a consequence it follows that $Range_{\partial D_n}(x)$ is even for all $x\in X_{\H}$ and $n \in \N$. We note that
$$Range_{A}(x)= Diameter(Image(\t x|_{A})).$$

The height function $h_x$ is subadditive, that is,
$$h_x(\m i,\m j)\leq h_x(\m i, \m k)+ h_x(\m k, \m j)$$
for all $x\in X_\H$ and $\m i ,\m j$ and $\m k \in \Z^d$. This is in contrast with the usual height function (as in \cite{chandgotia2013Markov} and \cite{peled2010high}) where there is an equality instead of the inequality. This raises some technical difficulties which are partly handled by the subadditive ergodic theorem.

The following terminology is not completely standard: Given a shift space $X$ a \emph{sub-cocycle} is a measurable map $c: X\times \Z^d \longrightarrow \N\cup \{0\}$ such that for all $\m i, \m j \in \Z^d$
$$c(x, \m i +\m j)\leq c(x, \m i)+ c(\sigma^{\m i}(x), \m j).$$
Sub-cocycles arise in a variety of situations; look for instance in \cite{Hammersleyfirst1965}. We are interested in the case $c(x, \m i)= h_x(\m 0, \m i)$ for all $x\in X_\H$ and $\m i \in \Z^d$. The measure of `rigidity' lies in the asymptotics of this sub-cocycle, the existence of which is provided by the subadditive ergodic theorem. Given a set $X$ if $f: X\longrightarrow \R$ is a function then let $f^+:=\max(0,f)$.

\begin{thm}[Subadditive Ergodic Theorem]\label{theorem:Subadditive_ergodic_theorem}\cite{walters-book}
Let $(X, \B, \mu)$ be a probability space and let $T: X\longrightarrow X$ be measure preserving. Let $\{f_n\}_{n=1}^\infty$ be a sequence of measurable functions $f_n: X\longrightarrow \R\cup \{-\infty\}$ satisfying the conditions:
\begin{enumerate}[(a)]
\item
$f_1^+ \in L^1(\mu)$
\item
for each $m$, $n \geq 1$, $f_{n+m }\leq f_n + f_m \circ T^n$ $\mu$-almost everywhere.
\end{enumerate}•
Then there exists a measurable function $f: X\longrightarrow \R\cup \{-\infty \}$ such that $f^+\in L^1(\mu)$, $f\circ T=f$, $\lim_{n\rightarrow \infty} \frac{1}{n}f_n =f$, $\mu$-almost everywhere
and
$$\lim_{n \longrightarrow \infty}\frac{1}{n}\int f_n d\mu = \inf_{n}\frac{1}{n}\int f_n d \mu= \int f d\mu.$$
\end{thm}

Given a direction $\m{ i} =(i_1, i_2, \ldots, i_d)\in \R^d$ let $\lfloor\m{ i}\rfloor=(\lfloor i_1\rfloor, \lfloor i_2\rfloor, \ldots, \lfloor i_d\rfloor)$. We define for all $x \in X_\H$ the \emph{slope of $x$ in the direction $\m{ i}$} as
$$sl_{\m {i}}(x):= \lim_{n \longrightarrow \infty}\frac{1}{n} h_x(\m 0, \lfloor n \m{ i}\rfloor)$$
whenever it exists.

If $\m i\in \Z^d$ we note that the sequence of functions $f_n: X_\H\longrightarrow \N\cup \{\m 0\}$ given by
$$f_n(x)=h_x(\m 0, n\m i)$$
satisfies the hypothesis of this theorem for any shift-invariant probability measure on $X_\H$: $|f_1|\leq \|\m i\|_1$ and the subadditivity condition in the theorem is just a restatement of the sub-cocycle condition described above, that is, if $T= \sigma^{\m i}$ then
$$f_{n+m }(x)= h_x(\m 0, (n+m)\m i)\leq h_x(\m 0, n \m i)+ h_{\sigma^{n \m i}x}(\m 0,m \m i ) =f_n(x) + f_m(T^n(x)).$$
The asymptotics of the height functions (or more generally the sub-cocycles) are a consequence of the subadditive ergodic theorem as we will describe next. In the following by an ergodic measure on $X_\H$, we mean a probability measure on $X_\H$ which is ergodic with respect to the $\Z^d$-shift action on $X_\H$.

\begin{prop}[Existence of Slopes]\label{prop:existence_of_slopes}
Let $\H$ be a connected four-cycle free graph and $\mu$ be an ergodic measure on $X_\H$. Then for all $\m{ i}\in \Z^d$
$$sl_{\m {i}}(x)=\lim_{n \longrightarrow \infty}\frac{1}{n} h_x({\m{0}}, n \m{ i})$$
exists almost everywhere and is independent of $x$. Moreover if $\m {i}= (i_1, i_2\ldots, i_d)$ then
$$sl_{\m {i}}(x)\leq \sum_{k=1}^d |i_k| sl_{\m {e}_k}(x).$$
\end{prop}
\begin{proof}
Fix a direction $\m{ i}\in \Z^d$. Consider the sequence of functions $\{f_n\}_{n=1}^\infty$ and the map $T: X_\H\longrightarrow X_\H$ as described above. By the subadditive ergodic theorem there exists a function $f: X_\H\longrightarrow \R\cup \{-\infty\}$ such that
$$\lim_{n \rightarrow \infty}\frac{1}{n}f_n=f\ almost\ everywhere.$$
Note that $f= sl_{\m{i}}$. Since for all $x\in X_\H$ and $n \in\N$, $0\leq f_n\leq n\|\mi\|_1$, $0\leq f(x)\leq \|\vec i\|_1$ whenever it exists. Fix any $\m{j}\in \Z^d$. Then
\begin{eqnarray*}
f_n(\sigma^{\m{j}}(x))&=& h_{\sigma^{\m{j}}(x)}({\m{0}}, n \m{i})= h_x(\m{j}, n \m{ i}+\m{ j})
\end{eqnarray*}
and hence
\begin{eqnarray*}
-h_x(\m{j}, {\m{0}}) + h_x({\m{0}}, n \m{i})- h_x(n \m{i}, n \m{i}+\m{j})&\leq& f_n(\sigma^{\m{j}}(x))\\
&\leq& h_x(\m{j},{\m{0}}) + h_x({\m{0}}, n \m{i})+ h_x(n \m{i}, n \m{i}+\m{j})
\end{eqnarray*}
implying
\begin{eqnarray*}
-2\|\m{j}\|_1+ f_n(x)\leq & f_n(\sigma^{\m{j}}(x))& \leq 2\|\m{j}\|_1+ f_n(x)
\end{eqnarray*}•
implying
$$f(x)=\lim_{n \longrightarrow \infty} \frac{1}{n} f_n(x)= \lim_{n \longrightarrow \infty}\frac{1}{n} f_n(\sigma^{\m{j}} x)= f(\sigma^{\m{j}}(x))$$
almost everywhere. Since $\mu$ is ergodic $sl_{\mi}= f$ is constant almost everywhere. Let $\m{i}^{(k)} = (i_1, i_2, \ldots, i_k, 0, \ldots, 0)\in \Z^d$. By the subadditive ergodic theorem
\begin{eqnarray*}
sl_{\m{i}}(x)= \int sl_{\m{i}}(x) d\mu&=& \lim_{n \longrightarrow \infty}\frac{1}{n}\int h_x({\m{0}}, n \m{i}) d\mu\\
&\leq&\sum_{k=1}^d \lim_{n \longrightarrow\infty}\frac{1}{n} \int h_{\sigma^{n \m{i}^{(k-1)}}(x)}({\m{0}}, ni_{k}\m{e}_k ) d \mu\\
&=&\sum_{k=1}^d \lim_{n \longrightarrow\infty}\frac{1}{n} \int h_x({\m{0}}, ni_{k}\m{e}_k ) d \mu\\
&\leq&\sum_{k=1}^d|i_k|\lim_{n \longrightarrow\infty}\frac{1}{n} \int h_x({\m{0}}, n\m{e}_k ) d \mu\\
&=&\sum_{k=1}^d |i_k| sl_{\m {e}_k}(x).
\end{eqnarray*}•	
almost everywhere.
\end{proof}
\begin{corollary}\label{corollary: existence_of _slopes_in_reality}
Let $\H$ be a connected four-cycle free graph. Suppose $\mu$ is an ergodic measure on $X_\H$. Then for all $\m{i}\in \R^d$
$$sl_{\m{i}}(x)=\lim_{n \longrightarrow \infty}\frac{1}{n} h_x({\m{0}}, \lfloor n \m{i}\rfloor)$$
exists almost everywhere and is independent of $x$. Moreover if $\m{i}= (i_1, i_2,\ldots, i_d)$ then
$$sl_{\m{i}}(x)\leq \sum_{k=1}^d |i_k| sl_{\m{e}_k}(x).$$
\end{corollary}
\begin{proof}
Let $\m{i}\in \Q^d$ and $N\in \N$ such that $N \m{i} \in \Z^d$. For all $n \in \N$ there exists $k \in \N\cup\{0\}$ and $0\leq m\leq N-1$ such that $n = kN+m$. Then
for all $x\in X_\H$
$$h_x({\m{0}}, k N\m{i})- N\|\m{i}\|_1\leq h_x({\m{0}}, \lfloor n\m{i} \rfloor)\leq h_x({\m{0}}, k N\m{i})+ N\|\m{i}\|_1$$
proving
$$sl_\mi(x)=\lim_{n\longrightarrow\infty}\frac{1}{n}h_x({\m{0}}, \lfloor n\m{ i} \rfloor) = \frac{1}{N}\lim_{k \longrightarrow \infty} \frac{1}{k}h_x({\m{0}}, k N\m{i}) = \frac{1}{N}sl_{N\m{i}}(x)$$
almost everywhere. Since $sl_{N\m{i}}$ is constant almost everywhere, we have that $sl_\mi$ is constant almost everywhere as well; denote the constant by $c_\mi$ . Also
$$sl_{\m{i}}(x)\leq \frac{1}{N}\sum _{l=1}^d|N i_l|sl_{\m{e}_l}(x)=\sum _{l=1}^d|i_l|sl_{\m{e}_l}(x).$$

Let $X\subset X_\H$ be the set of configurations $x$ such that
$$\lim_{n \longrightarrow \infty}\frac{1}{n} h_x({\m{0}}, \lfloor n \m{i}\rfloor)=c_\mi$$
for all $\mi \in \Q^d$. We have proved that $\mu(X)=1$.

Fix $x\in X$.
Let $\m i, \m{j}\in \R^d$ such that $\|\m{i}- \m{j}\|_1<\epsilon$. Then
$$\left|\frac{1}{n} h_x({\m{0}}, \lfloor n \m{i}\rfloor)-\frac{1}{n} h_x({\m{0}}, \lfloor n \m{j}\rfloor)\right|\leq\frac{1}{n}\|\lfloor n \m{i}\rfloor-\lfloor n \m{j}\rfloor\|_1\leq\epsilon+\frac{2d}{n}.$$
Thus we can approximate $\frac{1}{n} h_x({\m{0}}, \lfloor n \m{i}\rfloor)$ for $\mi \in \R^d$ by $\frac{1}{n} h_x({\m{0}}, \lfloor n \m{j}\rfloor)$ for $\mj \in \Q^d$ to prove that $\lim_{n \longrightarrow \infty}\frac{1}{n} h_x({\m{0}}, \lfloor n \m{i}\rfloor)$ exists for all $\m i \in \R^d$, is independent of $x\in X$ and satisfies
$$sl_{\m{i}} (x)\leq \sum _{k=1}^d|i_k|sl_{\m{e}_k}(x).$$
\end{proof}

The existence of slopes can be generalised from height functions to continuous sub-cocycles; the same proofs work:
\begin{prop}Let $c:X\times \Z^d \longrightarrow \R$ be a continuous sub-cocycle and $\mu$ be an ergodic measure on $X$. Then for all $\m{i}\in \R^d$
$$sl^c_{\m{i}}(x):=\lim_{n \longrightarrow \infty}\frac{1}{n} c(x, \lfloor n \m{i}\rfloor)$$
exists almost everywhere and is independent of $x$. Moreover if $\m{i}= (i_1, i_2\ldots, i_d)$ then
$$sl^c_{\m{i}}(x)\leq \sum_{k=1}^d |i_k| sl^c_{\m{e}_k}(x).$$
\end{prop}

Let $C_X$ be the space of continuous sub-cocycles on a shift space $X$. $C_X$ has a natural vector space structure: given $c_1, c_2\in C_X$, $(c_1 +\alpha c_2)$ is also a continuous sub-cocycle on $X$ for all $\alpha\in \R$ where addition and scalar multiplication is point-wise. The following is not hard to prove and follows directly from definition.
\begin{prop}\label{proposition: sub-cocycles under conjugacy}
Let $X, Y$ be conjugate shift spaces. Then every conjugacy $f: X \longrightarrow Y$ induces a vector-space isomorphism $ f^\star: C_Y\longrightarrow C_X$ given by
$$f^\star(c)(x, \m {i}):= c(f(x), \m{i})$$
for all $c\in C_Y$, $x\in X$ and $\m i \in \Z^d$. Moreover $sl^c_{\m i}(y)=sl^{f^\star(c)}_{\m i}(f^{-1}(y))$ for all $y\in Y$ and $\m i \in \R^d$ for which the slope $sl^c_{\m i}(y)$ exists.
\end{prop}
\section{Proofs of the Main Theorems} \label{section: Proof of the main theorems}
\begin{proof}[Proof of Theorem \ref{theorem: MRF fully supported }] If $\H$ is a single edge, then $X_\H$ is the orbit of a periodic configuration; the result follows immediately. Suppose this is not the case. The proof follows loosely the proof of Proposition \ref{proposition: periodicfoldentropy} and morally the ideas from \cite{lightwoodschraudnerentropy}: We prove existence of two kind of configurations in $X_\H$, ones which are `poor' (Lemma \ref{lemma:slope 1 is frozen}), in the sense that they are frozen and others which are `universal' (Lemma \ref{lemma:patching_various_parts}), for which the homoclinic class is dense.

Ideas for the following proof were inspired by discussions with Anthony Quas. A similar result in a special case is contained in Lemma 6.7 of \cite{chandgotia2013Markov}.

\begin{lemma}\label{lemma:slope 1 is frozen} Let $\H$ be a connected four-cycle free graph and $\mu$ be an ergodic probability measure on $X_\H$ such that $sl_{\m e_k}(x)=1$ almost everywhere for some $1\leq k \leq d$. Then $\mu$ is frozen and $h_\mu=0$.
\end{lemma}
\begin{proof}
Without loss of generality assume that $sl_{\m e_1}(x)=1$ almost everywhere. By the subadditivity of the height function for all $k, n \in \N$ and $x\in X_\H$ we know that
$$\frac{1}{kn}h_x(\m 0, kn\m{e}_1) \leq \frac{1}{kn}\sum_{m=0}^{n-1}h_x(km\m{e}_1, k(m+1)\m{e}_1)=\frac{1}{n}\sum_{m=0}^{n-1}\frac{1}{k}h_{\sigma^{km \m e_1} (x)}(\m 0, k\m{e}_1) \leq 1.$$
Since $sl_{\m{e}_1}(x)= 1$ almost everywhere, we get that
$$\lim_{n\longrightarrow \infty} \frac{1}{n}\sum_{m=0}^{n-1}\frac{1}{k}h_{\sigma^{km \m e_1} (x)}(\m 0, k\m{e}_1)=1$$
almost everywhere. By the ergodic theorem
$$\int \frac{1}{k}h_{x}(\m 0, k\m{e}_1) d \mu= 1.$$
Therefore $h_{x}(\m 0, k\m{e}_1)=k$ almost everywhere which implies that
\begin{equation}
h_x(\m i, \m i+k\m{e}_1)=k \label{eq:slopeoneheightconstantrise}
\end{equation}•
for all $\m i \in \Z^d$ and $k \in \N$ almost everywhere. Let $X\subset supp(\mu)$ denote the set of such configurations.

For some $n \in \N$ consider two patterns $a,b \in \L_{B_n\cup \partial_2 B_n}(supp(\mu))$ such that $a|_{\partial_2 B_n}= b|_{\partial_2 B_n}$. We will prove that then $a|_{B_n}= b|_{B_n}$. This will prove that $\mu$ is frozen, and $|\L_{B_n}(supp(\mu))|\leq|\L_{\partial_2 B_n}(supp(\mu))|\leq |\A|^{|\partial_2 B_n|}$ implying that $h_{top}(supp(\mu))=0$. By the variational principle this implies that $h_\mu=0$.

Consider $x, y \in X$ such that $x|_{B_n\cup \partial_2 B_n}= a$ and $y|_{B_n\cup \partial_2 B_n}= b$. Noting that $\partial_2 B_n$ is connected, by Corollary \ref{corollary:covering_space_lifting_homoclinic} we can choose lifts $\t x, \t y\in X_{E_\H}$ such that $\t x|_{\partial_2 B_n}= \t y|_{\partial_2 B_n}$. Consider any $\m i \in B_n$ and choose $k\in - \N$ such that $\m i + k \m e_1, \m i + (2n+2+k)\m e_1 \in \partial B_n$. Then by Equation \ref{eq:slopeoneheightconstantrise} $d_{E_\H}(\t x_{\m i + k \m e_1}, \t x_{\m i + (2n+2+k) \m e_1})= 2n+2$. But
$$(\t x_{\m i + k \m e_1},\t x_{\m i + (k+1)\m e_1}, \ldots,\t x_{\m i + (2n+2+k) \m e_1} )\text{ and }$$
$$(\t y_{\m i + k \m e_1},\t y_{\m i + (k+1)\m e_1}, \ldots,\t y_{\m i + (2n+2+k) \m e_1} )$$
are walks of length $2n+2$ from $\t x_{\m i + k \m e_1}$ to $\t x_{\m i + (2n+2+k) \m e_1}$. Since $E_\H$ is a tree and the walks are of minimal length, they must be the same. Thus $\t x|_{B_n}=\t y|_{B_n}$. Taking the image under the map $\pi$ we derive that
$$a|_{B_n}=x|_{B_n}=y|_{B_n}= b|_{B_n}.$$
\end{proof}

This partially justifies the claim that steep slopes lead to greater `rigidity'. We are left to analyse the case where the slope is submaximal in every direction. As in the proof of Proposition 7.1 in \cite{chandgotia2013Markov} we will now prove a certain mixing result for the shift space $X_\H$.

\begin{lemma}\label{lemma:patching_various_parts} Let $\H$ be a connected four-cycle free graph and $|\H|= r$. Consider any $x\in X_\H$ and some $y \in X_\H$ satisfying $Range_{\partial D_{(d+1)n+3r+k}}(y)\leq 2k$ for some $n \in \N$. Then
\begin{enumerate}
\item\label{case:not_bipartite}
If either $\H$ is not bipartite or $x_{\m 0}, y_{\m 0}$ are in the same partite class of $\H$ then there exists $z\in X_\H$ such that
$$z_{\m i}=
\begin{cases}
x_{\m i}& \ if \ \m i \in D_n\\
y_{\m i} &\ if \ \m i \in D_{(d+1)n+3r+k}^c.
\end{cases}$$
\item \label{case:bipartite}
If $\H$ is bipartite and $x_{\m 0}, y_{\m 0}$ are in different partite classes of $\H$ then there exists $z\in X_\H$ such that
$$z_{\m i}=
\begin{cases}
x_{\m i+\m e_1} &if \ \m i \in D_n\\
y_{\m i} & if \ \m i \in D_{(d+1)n+3r+k}^c.
\end{cases}$$
\end{enumerate}•
\end{lemma}
The separation $dn+3r+k$ between the induced patterns of $x$ and $y$ is not optimal, but sufficient for our purposes.
\begin{proof} We will construct the configuration $z$ only in the case when $\H$ is not bipartite. The construction in the other cases is similar; the differences will be pointed out in the course of the proof.
\begin{enumerate}
\item
\textbf{Boundary patterns with non-maximal range to monochromatic patterns inside.}
Let $\t y$ be a lift of $y$ and $\T^\prime$ be the image of $\t y|_{ D_{(d+1)n+3r+k+1}}$. Let $\T$ be a minimal subtree of $E_\H$ such that
$$Image(\t y|_{\partial D_{(d+1)n+3r+k}})\subset \T\subset \T^\prime.$$
Since $Range_{\partial D_{(d+1)n+3r+k}}(y)\leq 2k$, $diameter(\T)\leq 2k$. By Proposition \ref{proposition:folding trees into other trees} there exists a graph homomorphism $f:\T^\prime \longrightarrow \T$ such that $f|_\T$ is the identity. Consider the configuration $\t y^1$ given by
$$\t y^1_{\m i}= \begin{cases}
f(\t y_{\m i}) &\text{ if }\m i \in D_{(d+1)n+3r+k+1}\\
\t y_{\m i} &\text{ otherwise.}
\end{cases}•$$
The pattern
$$\t y^1|_{D_{(d+1)n+3r+k+1}}\in \L_{D_{(d+1)n+3r+k+1}}(X_{\T})\subset \L_{D_{(d+1)n+3r+k+1}}(X_{E_\H}).$$
Moreover since $f|_\T$ is the identity map,
$$\t y^1|_{D_{(d+1)n+3r+k}^c}=\t y|_{D_{(d+1)n+3r+k}^c}\in \L_{D_{(d+1)n+3r+k}^c}(X_{E_\H}).$$
Since $X_{E_\H}$ is given by nearest neighbour constraints $\t y^1\in X_{E_\H}$.

Recall that the fold-radius of a nearest neighbour shift of finite type (in our case $X_\T$) is the total number of full config-folds required to obtain a stiff shift. Since $diameter(\T)\leq 2k$ the fold-radius of $X_{\T}\leq k$. Let a stiff shift obtained by a sequence of config-folds starting at $X_{\T}$ be denoted by $Z$. Since $\T$ folds into a graph consisting of a single edge, $Z$ consists of two checkerboard patterns in the vertices of an edge in $\T$, say $\t v_1$ and $\t v_2$. Corresponding to such a sequence of full config-folds, we had defined in Section \ref{section:Folding, Entropy Minimality and the Pivot Property} the outward fixing map $O_{X_\T, (d+1)n+3r+k}$. By Proposition \ref{prop: folding_ to _ stiffness_fixing_a_set} the configuration $O_{X_\T,(d+1)n+3r+k}(\t y^1)\in X_{E_{\H}}$ satisfies
\begin{eqnarray*}
O_{X_\T,(d+1)n+3r+k}(\t y^1)|_{D_{(d+1)n+3r+1}}\in \L_{D_{(d+1)n+3r+1}}(Z) \\
O_{X_\T,(d+1)n+3r+k}(\t y^1)|_{D_{(d+1)n+3r+k}^c}=\t y^1|_{D_{(d+1)n+3r+k}^c}=\t y|_{D_{(d+1)n+3r+k}^c}.
\end{eqnarray*}
\noindent Note that the pattern $O_{X_\T,(d+1)n+3r+k}(\t y^1)|_{\partial D_{(d+1)n+3r}}$ uses a single symbol, say $\t v_1$. Let $\pi (\t v_1)= v_1$. Then the configuration $y^\prime= \pi(O_{X_\T,(d+1)n+3r+k}(\t y^1))\in X_\H$ satisfies
\begin{eqnarray*}
y^\prime|_{\partial D_{(d+1)n+3r}} &=& v_1\\
y^\prime|_{D_{(d+1)n+3r+k}^c}&=&y|_{D_{(d+1)n+3r+k}^c}.
\end{eqnarray*}•
\item
\textbf{Constant extension of an admissible pattern.}
Consider some lift $\t x$ of $x$. We begin by extending $\t x|_{B_n}$ to a periodic configuration $\t x^1\in X_{E_\H}$. Consider the map $f: [-n, 3n]\longrightarrow [-n, n]$ given by
\begin{equation*}
f(k)=\begin{cases}
k &\text{ if } k \in [-n,n]\\
2n-k &\text{ if }k \in [n,3n].
\end{cases}•
\end{equation*}
\noindent Then we can construct the pattern $\t a\in \L_{[-n, 3n]^d}(X_{E_\H})$ given by
$$\t a_{i_1, i_2, \ldots i_d}= \t x_{f(i_1), f(i_2), \ldots, f(i_d)}.$$
Given $k, l \in [-n, 3n]$ if $|k-l|=1$ then $|f(k)-f(l)|=1$. Thus $\t a$ is a locally allowed pattern in $X_{E_\H}$. Moreover since $f(-n)= f(3n)$ the pattern $\t a$ is `periodic', meaning,
$$\t a_{i_1, i_2,\ldots, i_{k-1}, -n, i_{k+1}, \ldots, i_d }= \t a_{i_1, i_2, \ldots, i_{k-1}, 3n, i_{k+1}, \ldots, i_d }$$
for all $i_1, i_2, \ldots, i_d \in [-n,3n]$. Also $\t a|_{B_n}=\t x|_{B_n}$. Then the configuration $\t x^1$ obtained by tiling $\Z^d$ with $\t a|_{[-n,3n-1]^d}$, that is,
$$\t x^1_{\m i}= \t a_{(i_1\!\!\!\mod 4n,\ i_2\!\!\!\mod 4n,\ \ldots,\ i_d\!\!\!\mod 4n)-(n, n, \ldots, n)}\text{ for all }\m i \in \Z^d$$
is an element of $X_{E_\H}$. Moreover $\t x^1|_{B_n}= \t a|_{B_n}= \t x|_{B_n}$ and $Image(\t x^1)= Image(\t x|_{B_n})$. Since $diameter(B_n)=2dn$, $diameter(Image(\t x^1))\leq 2dn$. Let $\t \T= Image(\t x^1)$. Then the fold-radius of $X_{\t \T}$ is less than or equal to $dn$. Let a stiff shift obtained by a sequence of config-folds starting at $X_{\t \T}$ be denoted by $Z'$. Since $\t \T$ folds into a graph consisting of a single edge, $Z^\prime$ consists of two checkerboard patterns in the vertices of an edge in $\t T$, say $\t w_1$ and $\t w_2$. Then by Proposition \ref{prop: folding_ to _ stiffness_fixing_a_set}
\begin{eqnarray*}
I_{X_{\t\T},n}(\t x^1)|_{D_n}= \t x^1|_{D_n}= \t x|_{D_n}\\
I_{X_{\t\T},n}(\t x^1)|_{D_{(d+1)n-1}^c} \in \L_{D_{(d+1)n-1}^c}(Z^\prime).
\end{eqnarray*}
\noindent We note that $I_{X_{\t\T},n}(\t x^1)|_{\partial D_{(d+1)n-1}}$ consists of a single symbol, say $\t w_1$. Let $\pi(\t w_1)= w_1$. Then the configuration $x^\prime=\pi(I_{X_{\t\T},n}(\t x^1)) \in X_{\H}$ satisfies
\begin{eqnarray*}
x^\prime|_{D_n}= x|_{D_n}\text{ and}\\
x^\prime|_{\partial D_{(d+1)n-1}}=w_1.
\end{eqnarray*}•
\item \textbf{Patching of an arbitrary pattern inside a configuration with non-maximal range.}
We will first prove that there exists a walk on $\H$ from $w_1$ to $v_1$, $((w_1= u_1), u_2, \ldots, (u_{3r+2}= v_1))$.
Since the graph is not bipartite, it has a cycle $p_1$ such that $|p_1|\leq r-1$ and is odd. Let $v^\prime$ be a vertex in $p_1$. Then there exist walks $p_2$ and $p_3$ from $w_1$ to $v^\prime$ and from $v^\prime$ to $v_1$ respectively such that $|p_2|, |p_3|\leq r-1$. Consider any vertex $w^\prime\sim_\H v_1$. If
$3r+1-|p_2|- |p_3|$ is even then the walk
$$p_2\star p_3 (\star (v_1, w^\prime, v_1))^{\frac{3r+1-|p_2|- |p_3|}{2}}$$
and if not, then the walk
$$p_2\star p_1 \star p_3 (\star (v_1, w^\prime, v_1))^{\frac{3r+1-|p_1|-|p_2|- |p_3|}{2}}$$
is a walk of length $3r+1$ in $\H$ from $w_1$ to $v_1$. This is the only place where we use the fact that $\H$ is not bipartite. If it were bipartite, then we would require that $x^\prime_{\m 0}$ and $y^\prime_{\m 0}$ have to be in the same partite class to construct such a walk.

Given such a walk the configuration $z$ given by
\begin{eqnarray*}
z|_{D_{(d+1)n}}&=& x^\prime|_{D_{(d+1)n}}\\
z|_{D^c_{(d+1)n +3r}}&= &y^\prime|_{D^c_{(d+1)n +3r}}\\
z|_{\partial D_{(d+1)n+i-2}}&=& u_i\text{ for all } 1\leq i \leq 3r+2
\end{eqnarray*}
\noindent is an element of $X_\H$ for which $z|_{D_n}=x^\prime|_{D_n}=x|_{D_n}$ and $z|_{D_{(d+1)n+3r+k}^c}=y^\prime|_{D_{(d+1)n+3r+k}^c}=y|_{D_{(d+1)n+3r+k}^c}.$
\end{enumerate}•

\end{proof}

We now return to the proof of Theorem \ref{theorem: MRF fully supported }. Let $\mu$ be an ergodic probability measure adapted to $X_\H$ with positive entropy.

Suppose $sl_{\m e_i}(x)= \theta_i$ almost everywhere. By Lemma \ref{lemma:slope 1 is frozen}, $\theta_i<1$ for all $1\leq i \leq d$. Let $\theta= \max_i \theta_i$ and $0<\epsilon<\frac{1}{4}\left(1- \theta\right)$. Denote by $S^{d-1}$, the sphere of radius $1$ in $\R^d$ for the $l^1$ norm. By Corollary \ref{corollary: existence_of _slopes_in_reality} for all $\m{v}\in S^{d-1}$
$$\lim_{n\longrightarrow \infty }\frac{1}{n}h_x({\vec{0}}, \lfloor n \m{v}\rfloor) \leq \theta$$
almost everywhere. Since $S^{d-1}$ is compact in $\R^d$ we can choose a finite set $\{\m{v}_1, \m{v}_2, \ldots, \m{v}_t\} \subset S^{d-1}$ such that for all $\m{v}\in S^{d-1}$ there exists some $1\leq i\leq t$ satisfying $\|\m{v}_i -\m v\|_1<\epsilon$. By Egoroff's theorem \cite{Follandreal1999} given $\epsilon$ as above there exists $N_0\in \N$ such that for all $n\geq N_0$ and $1\leq i \leq t$
\begin{equation}
\mu(\{x\in X_\H\:|\:h_x({\vec{0}}, \lfloor n \m{v}_i\rfloor)\leq n\theta + n\epsilon\ for\ all\ 1\leq i\leq t\}) >1-\epsilon.\label{equation:uniform_continuity_of_heights}
\end{equation}•
Let $\m{v} \in \partial D_{n-1}$ and $1\leq i_0\leq t$ such that
$\|\frac{1}{n}\m{v}-\m{v}_{i_0}\|_1<\epsilon$. If for some $x\in X_\H$ and $n \in \N$
$$h_x({\vec{0}}, \lfloor n \m{v}_{i_0}\rfloor)\leq n\theta + n\epsilon$$
then
$$h_x({\vec{0}}, \lfloor \m{v}\rfloor)\leq h_x({\vec{0}}, \lfloor n \m{v}_{i_0}\rfloor) +\lceil n \epsilon \rceil \leq n\theta + 2n\epsilon+1.$$
By Inequality \ref{equation:uniform_continuity_of_heights} we get
$$\mu\left(\{x\in X_\H\:|\: h_x\left({\vec{0}}, \lfloor \m{v}\rfloor\right)\leq n\theta + 2n\epsilon+1\ for\ all\ \m{v}\in \partial D_{n-1}\}\right) >1-\epsilon$$
for all $n\geq N_0$. Therefore for all $n\geq N_0$ there exists $x^{(n)}\in supp(\mu) $ such that
$$Range_{\partial D_{n-1}}\left(x^{(n)}\right)\leq 2n\theta + 4 n \epsilon +2< 2n(1- \epsilon)+2.$$
Let $x \in X_\H$ and $n_0\in \N$. It is sufficient to prove that $\mu([x]_{D_{n_0-1}})>0$. Suppose $r:=|\H|$. Choose $k \in \N$ such that
\begin{eqnarray*}
n_0(d+1)+3r+k+1&\geq&N_0\\
2\left(n_0(d+1)+3r+k+1\right)(1-\epsilon)+2&\leq&2k.
\end{eqnarray*}•
Then by Lemma \ref{lemma:patching_various_parts} there exists $z\in X_\H$ such that either
\begin{equation*}
z_\mj=
\begin{cases}
x_\mj \quad\quad\quad\quad \quad\quad\: &if \ \mj \in D_{n_0}\\
x^{\left(n_0(d+1)+3r+k+1\right)}_\mj \ &if \ \mj \in D_{n_0(d+1)+3r+k}^c
\end{cases}
\end{equation*}•
or
\begin{equation*}
z_\mj=
\begin{cases}
x_{\mj +\m e_1}\quad\quad\quad\quad \quad\quad \:& if \ \mj \in D_{n_0}\\
x^{\left(n_0(d+1)+3r+k+1\right)}_\mj \ &if \ \mj \in D_{n_0(d+1)+3r+k}^c.
\end{cases}
\end{equation*}•
In either case $(z, x^{\left(n_0(d+1)+3r+k+1\right)})\in \Delta_{X_\H}$. Since $\mu$ is adapted to $X_\H$, $z\in supp(\mu)$. In the first case we get that $\mu([x]_{D_{n_0-1}})=\mu([z]_{D_{n_0-1}})>0$. In the second case we get that $$\mu([x]_{D_{n_0-1}})=\mu(\sigma^{\m e_1}([x]_{D_{n_0-1}}))=\mu([z]_{D_{(n_0-1)}-\m e_1})>0.$$
This completes the proof.

\end{proof}

Every shift space conjugate to an entropy minimal shift space is entropy minimal. However a shift space $X$ which is conjugate to $X_\H$ for $\H$ which is connected and four-cycle free need not even be a hom-shift. By following the proof carefully it is possible to extract a condition for entropy minimality which is conjugacy-invariant:

\begin{thm}\label{theorem:conjugacy_invariant_entropy minimality condition}
Let $X$ be a shift of finite type and $c$ a continuous sub-cocycle on $X$ with the property that $c(\cdot, \mi)\leq \|\mi\|_1$ for all $\mi \in \Z^d$. Suppose every ergodic probability measure $\mu$ adapted to $X$ satisfies:
\begin{enumerate}
\item
If $sl^c_{\m e_i}(x)=1$ almost everywhere for some $1\leq i \leq d$ then $h_\mu< h_{top}(X)$.
\item
If $sl^c_{\m e_i}(x)<1$ almost everywhere for all $1\leq i \leq d$ then $supp(\mu)=X$.
\end{enumerate}
then $X$ is entropy minimal.
\end{thm}

Here is a sketch: By Proposition \ref{proposition:entropyviamme} and Theorems \ref{thm:equiGibbs}, \ref{theorem: ergodic decomposition of markov random fields} it is sufficient to prove that every ergodic measure of maximal entropy is fully supported. If $X$ is a shift of finite type satisfying the hypothesis of Theorem \ref{theorem:conjugacy_invariant_entropy minimality condition} then it is entropy minimal because every ergodic measure of maximal entropy of $X$ is an ergodic probability measure adapted to $X$; its entropy is either smaller than $h_{top}(X)$ or it is fully supported. To see why the condition is conjugacy invariant suppose that $f:X\longrightarrow Y$ is a conjugacy and $c\in C_Y$ satisfies the hypothesis of the theorem. Then by Proposition \ref{proposition: sub-cocycles under conjugacy} it follows that ${f^\star}(c)\in C_X$ satisfies the hypothesis as well.

\begin{proof}[Proof of Theorem \ref{theorem: pivot property for four cycle free}] By Proposition \ref{proposition: pivot for disconnected} we can assume that $\H$ is connected. Consider some $(x, y)\in \Delta_{X_\H}$. By Corollary \ref{corollary:covering_space_lifting_homoclinic} there exist $(\t x, \t y)\in \Delta_{X_{E_\H}}$ such that $\pi(\t x)= x$ and $\pi(\t y)=y$. It is sufficient to prove that there is a chain of pivots from $\t x$ to $\t y$. We will proceed by induction on $\sum_{\mi \in \Z^d} d_{E_\H}(\t x_{\mi}, \t y_\mi)$. The induction hypothesis (on $M$) is : If $\sum_{\mi \in \Z^d} d_{E_\H}(\t x_\mi, \t y_\mi)= 2M$ then there exists a chain of pivots from $\t x$ to $\t y$.

We note that $d_{E_\H}(\t x_\mi, \t y_\mi)$ is even for all $\mi \in \Z^d$ since there exists $\mi^\p\in \Z^d$ such that $\t x_{\mi^\p}= \t y_{\mi^{\p}}$ and hence $\t x_\mi$ and $\t y_\mi$ are in the same partite class of $E_\H$ for all $\mi \in \Z^d$.

The base case $(M=1)$ occurs exactly when $\t x$ and $\t y$ differ at a single site; there is nothing to prove in this case. Assume the hypothesis for some $M\in \N$.

Consider $(\t x, \t y)\in \Delta_{X_{E_\H}}$ such that
$$\sum_{\mi \in \Z^d} d_{E_\H}(\t x_\mi, \t y_\mi)=2M+2.$$
Let
$$B=\{\mj \in \Z^d\:|\: \t x_\mj \neq \t y_\mj\}$$
and a vertex $\t v\in E_\H$. Without loss of generality we can assume that
\begin{equation}
\max_{\mi \in B} d_{E_\H}(\t v, \t x_\mi)\geq \max_{\mi \in B} d_{E_\H}(\t v, \t y_\mi).\label{equation:assumption_for_pivot}
\end{equation}•
Consider some $\mi_0 \in B$ such that
$$d_{E_\H}(\t v, \t x_{\mi_0})= \max_{\mi \in B}d_{E_\H}(\t v, \t x_{\mi}).$$
Consider the shortest walks $(\t v= \t v_1, \t v_2, \ldots, \t v_n=\t x_{\mi_0})$ from $\t v$ to $\t x_{\mi_0}$ and $(\t v= \t v^\p_1, \t v^\p_2, \ldots, \t v^\p_{n^\p}=\t y_{\mi_0})$ from $\t v$ to $\t y_{\mi_0}$. By Assumption \ref{equation:assumption_for_pivot}, $n^\p\leq n$. Since these are the shortest walks on a tree, if $\t v^\p_k=\t v_{k^\p}$ for some $1\leq k\leq n^\p$ and $1\leq k^{\p} \leq n$ then $k =k^{\p}$ and $\t v_l = \t v_l^\p$ for $1\leq l \leq k$. Let
$$k_0= \max\{1\leq k\leq n^\p\:|\: \t v_k^\p= \t v_k\}.$$
Then the shortest walk from $\t x_{\mi_0}$ to $\t y_{\mi_0}$ is given by $\t x_{\mi_0}=\t v_n, \t v_{n-1}, \t v_{n-2}, \ldots, \t v_{k_0}, \t v^\p_{k_0+1}, \ldots, \t v^\p_{n^\p}= \t y_{\mi_0}$.

We will prove for all $\m i \sim \m i_{0}$, $\t x_{\m i}= \t v_{n-1}$. This is sufficient to complete the proof since then the configuration
$$\t x^{(1)}_{\mj}=
\begin{cases}
\t x_\mj \quad\ \:if\ \mj \neq \mi_0\\
\t v_{n-2}\ \ if\ \mj = \mi_0,
\end{cases}$$
\noindent is an element of $X_{E_\H}$, $(\t x,\t x^{(1)})$ is a pivot and
$$n+n^\p -2 k_0 -2=d_{E_\H}{(\t x^{(1)}_{\mi_0}, \t y_{\mi_0})}< d_{E_\H}(\t x_{\mi_0}, \t y_{\mi_0})=n+n^\p -2 k_0$$
\noindent giving us a pair $(\t x^{(1)}, \t y)$ such that
$$\sum_{\mi \in \Z^d} d_{E_\H}(\t x^{(1)}_\mi, \t y_\mi)=\sum_{\mi \in \Z^d} d_{E_\H}(\t x_\mi, \t y_\mi)-2= 2M$$

to which the induction hypothesis applies. There are two possible cases:
\begin{enumerate}
\item
$\mi \in B$: Then $d_{E_\H}(\t v, \t x_{\m i})=d_{E_\H}(\t v, \t x_{\m i_0})-1$ and $\t x_\mi\sim_{E_\H} \t x_{\m i_0}$. Since $E_\H$ is a tree, $\t x_\mi= \t v_{n-1}$.

\item $\mi \notin B$: Then $\t x_\mi= \t y_\mi$ and we get that $d_{E_\H}(\t x_{\mi_0}, \t y_{\mi_0})=2$. Since $\t x_\mi\sim_{E_\H} \t x_{\mi_0}$, the shortest walk joining $\t v$ and $\t x_{\mi}$ must either be $\t v= \t v_1, \t v_2, \ldots, \t v_{n-1}= \t x_{\mi}$ or $\t v= \t v_1, \t v_2,\ldots,\t v_{n}= \t x_{\mi_0}, \t v_{n+1}= \t x_{\mi}$. We want to prove that the former is true. Suppose not.

Since $\t y_{\mi_0}\sim_{E_\H} \t x_{\mi}$ and $\mi_0\in B$, the shortest walk from $\t v$ to $\t y_{\mi_0}$ is $\t v= \t v_1, \t v_2,\ldots,\t v_{n}= \t x_{\mi_0}, \t v_{n+1}= \t x_{\mi}, \t v_{n+2}=\t y_{\mi_0} $. This contradicts Assumption \ref{equation:assumption_for_pivot} and completes the proof.

\end{enumerate}

\end{proof}

\section{Further Directions}
\subsection{Getting Rid of the Four-Cycle Free Condition}

In the context of the results in this paper, the four-cycle free condition seems a priori artificial; we feel that in many cases it is a mere artifact of the proof. To the author, getting rid of this condition is an important and interesting topic for future research. Here we will illustrate what goes wrong when we try to apply our proofs for the simplest possible example with four-cycles, that is, $C_4$.

We have shown (Example \ref{Example: Folds to an edge}) that $X_{C_4}$ satisfies the hypothesis of Propositions \ref{proposition: periodicfoldentropy} and \ref{proposition: frozenfoldpivot} and thus it also satisfies the conclusions of Theorems \ref{theorem: MRF fully supported } and \ref{theorem: pivot property for four cycle free}. The proofs of Theorems \ref{theorem: MRF fully supported } and \ref{theorem: pivot property for four cycle free} however rely critically on the existence of lifts to the universal cover, that is, Proposition \ref{proposition:covering_space_lifting}. However the conclusion of this proposition does not hold for $X_{C_4}$: The universal cover of $C_4$ is $\Z$ and the corresponding covering map $\pi: \Z \longrightarrow C_4$ is given by $\pi(i)= i \mod 4$. By the second remark following Theorem 4.1 in \cite{chandgotia2013Markov} it follows that the induced map $\pi: X_{\Z}\longrightarrow X_{C_4}$ is not surjective disproving the conclusion of Proposition \ref{proposition:covering_space_lifting} for $X_{C_4}$.

\subsection{Identification of Hom-Shifts}
\noindent\textbf{Question 1:} Given a shift space $X$, are there some nice decidable conditions which imply that $X$ is conjugate to a hom-shift?

Being conjugate to a hom-shift lays many restrictions on the shift space, for instance on its periodic configurations. Consider a conjugacy $f:X\longrightarrow X_\H$ where $\H$ is a finite undirected graph. Let $Z\subset X_\H$ be the set of configurations invariant under $\{\sigma^{2\m e_i}\}_{i=1}^d$. Then there is a bijection between $Z$ and $\L_A(X_\H)$ where $A$ is the rectangular shape
$$A:=\{\sum_{i=1}^{d}\delta_i \m e_i\:|\: \delta_i\in \{0,1\}\}$$
because every pattern in $\L_A(X_\H)$ extends to a unique configuration in $Z$. More generally given a graph $\H$ it is not hard to compute the number of periodic configurations for a specific finite-index subgroup of $\Z^d$. Moreover periodic points are dense in these shift spaces and there are algorithms to compute approximating upper and lower bounds of their entropy \cite{symmtricfriedlan1997,louidor2010improved}. Hence the same then has to hold for the shift space $X$ as well. We are not familiar with nice decidable conditions which imply that a shift space is conjugate to a hom-shift.

\subsection{Hom-Shifts and Strong Irreducibility}\label{subsection: homSI}

\noindent\textbf{Question 2:} Which hom-shifts are strongly irreducible?

We know two such conditions:
\begin{enumerate}
\item\cite{brightwell2000gibbs}
If $\H$ is a finite graph which folds into $\H^\prime$ then $X_\H$ is strongly irreducible if and only if $X_{\H'}$ is strongly irreducible. This reduces the problem to graphs $\H$ which are stiff. For instance if $\H$ is dismantlable, then $X_\H$ is strongly irreducible.
\item\cite{Raimundo2014}
$X_\H$ is single site fillable. A shift space $X_\F\subset \A^{\Z^d}$ is said to be \emph{single site fillable} if for all patterns $a\in \A^{\partial\{\m 0\}}$ there exists a locally allowed pattern in $X_\F$, $b\in \A^{D_1}$ such that
$b|_{\partial\{\m 0\}}=a$. In case $X_\F= X_\H$ for some graph $\H$ then it is single site fillable if and only if given vertices $v_1, v_2, \ldots, v_{2d}\in \H$ there exists a vertex $v\in \H$ adjacent to all of them.
\end{enumerate}
It follows that $X_{K_5}$ is single site fillable and hence strongly irreducible for $d=2$. In fact strong irreducibility has been proved in \cite {Raimundo2014} for shifts of finite type with a weaker mixing condition called TSSM. This does not cover all possible examples. For instance it was proved in \cite{Raimundo2014} that $X_{K_4}$ is strongly irreducible for $d=2$ even though it is not TSSM and $K_4$ is stiff. We do not know if it is possible to verify whether a given hom-shift is TSSM.

\subsection{Hom-Shifts and Entropy Minimality}\textbf{Question 3:} Given a finite connected graph $\H$ when is $X_\H$ entropy minimal?

We have provided some examples in the paper:
\begin{enumerate}
\item
$\H$ can be folded to a single vertex with a loop or a single edge. (Proposition \ref{proposition: periodicfoldentropy})
\item $\H$ is four-cycle free. (Theorem \ref{theorem:four cycle free entropy minimal})
\end{enumerate}•
Again this does not provide the full picture. For instance $X_{K_4}$ is strongly irreducible when $d=2$ and hence entropy minimal even though $K_4$ is stiff and not four-cycle free.
A possible approach might be via identifying the right sub-cocycle and Theorem \ref{theorem:conjugacy_invariant_entropy minimality condition}.

\noindent\textbf{Conjecture:} Let $d=2$ and $\H$ be a finite connected graph. Then $X_\H$ is entropy minimal.

\subsection{Hom-Shifts and the Pivot Property}\label{subsection: Hom-shifts and the pivot property}
We have given a list of examples of graphs $\H$ for which the shift space $X_\H$ has the pivot property in Section \ref{section: the pivot property}. In this paper we have provided two further sets of examples:
\begin{enumerate}
\item
$\H$ can be folded to a single vertex with a loop or a single edge. (Proposition \ref{proposition: frozenfoldpivot})
\item
$\H$ is four-cycle free. (Theorem \ref{theorem: pivot property for four cycle free})
\end{enumerate}

We saw in Section \ref{section: the pivot property} that $X_{K_4}, X_{K_5}$ do not have the pivot property when $d=2$. However they do satisfy a weaker property which we will describe next.

A shift space $X$ is said to have the \emph{generalised pivot property} if there is an $r\in \N$ such that for all $(x,y)\in \Delta_X$ there exists a chain $x^1=x, x^2, x^3, \ldots, y=x^n\in X$ such that $x^i$ and $x^{i+1}$ differ at most on some translate of $D_r$.

It can be shown that any nearest neighbour shift of finite type $X\subset \A^\Z$ has the generalised pivot property. In higher dimensions this is not true without any hypothesis; look for instance in Section 9 in \cite{chandgotia2013Markov}. It is not hard to prove that any single site fillable nearest neighbour shift of finite type has the generalised pivot property. This can be generalised further: in \cite{Raimundo2014} it is proven that every shift space satisfying TSSM has the generalised pivot property.

\noindent\textbf{Question 4:} For which graphs $\H$ does $X_\H$ satisfy the pivot property? What about the generalised pivot property?

\section{Acknowledgments}
I would like to thank my advisor, Prof. Brian Marcus for dedicated reading of a million versions of this paper, numerous suggestions, insightful discussions and many other things. The line of thought in this paper was begot in discussions with Prof. Tom Meyerovitch, his suggestions and remarks have been very valuable to me. I will also like to thank Prof. Ronnie Pavlov, Prof. Sam Lightwood, Prof. Michael Schraudner, Prof. Anthony Quas, Prof. Klaus Schmidt, Prof. Mahan Mj, Prof. Peter Winkler and Raimundo Brice\~no for giving a patient ear to my ideas and many useful suggestions. Lastly, I will like to thank Prof. Jishnu Biswas; he had introduced me to universal covers, more generally to the wonderful world of algebraic topology. This research was partly funded by the Four-Year Fellowship at the University of British Columbia. Lastly I would like to thank the anonymous referee for giving many helpful comments and corrections largely improving the quality of the paper.

\bibliographystyle{abbrv}
\bibliography{4Cfree}
\end{document}